\pdfoutput=1
\documentclass[preprint,12pt]{elsarticle}

\usepackage[utf8]{inputenc}

\usepackage{amssymb,amstext,amsmath,amsthm}
\usepackage{bm}

\usepackage{graphicx,wrapfig}
\usepackage[hang]{subfigure}
\usepackage[all]{xy}

\usepackage{url}
\usepackage{hyperref}
\usepackage{theomac}
\usepackage{amsmath,amsthm}
\usepackage{theomac}
\newcommand{\lemlab}[1]{\label{lemma:#1}}
\newcommand{\theolab}[1]{\label{theo:#1}}

\newcommand{\corlab}[1]{\label{cor:#1}}

\newcommand{\seclab}[1]{\label{section:#1}}

\newcommand{\proplab}[1]{\label{prop:#1}}
\newcommand{\conjlab}[1]{\label{conj:#1}}

\newcommand{\lemref}[1]{Lemma \ref{lemma:#1}}
\newcommand{\theoref}[1]{Theorem \ref{theo:#1}}
\newcommand{\cororef}[1]{Corollary \ref{cor:#1}}

\newcommand{\figref}[1]{Figure \ref{fig:#1}}
\renewcommand{\eqref}[1]{(\ref{eq:#1})}
\newcommand{\secref}[1]{Section \ref{section:#1}}

\newcommand{\propref}[1]{Proposition \ref{prop:#1}}
\newcommand{\conjref}[1]{Conjecture \ref{conj:#1}}

\newtheoremWithMacro{theorem}{Theorem}
\newtheoremWithMacro{lemma}{Lemma}[section]
\newtheoremWithMacro{prop}{Proposition}
\newtheorem{cor}[lemma]{Corollary}

\newtheorem{conj}[lemma]{Conjecture}
\newtheorem{question}[lemma]{Question}

\def\lemmaD#1{
\begin{lemma}
\label{lem:#1}
}

\newcommand{\theoremD}[1]{
\begin{theorem}
\label{theorem:#1}
}

\newcommand{\factD}[1]{
\begin{fact}
\label{fact:#1}
}

\newcommand{\corD}[1]{
\begin{cor}
\label{cor:#1}
}

\newcommand{\card}[1]{\ensuremath{\left\vert #1 \right\vert}}
\newcommand{\iprod}[2]{\ensuremath{\left\langle {#1}, {#2}\right\rangle}}
\renewcommand{\vec}[1]{\mathbf{#1}}

\newcommand\R{\text{$\mathbb{R}$}}

\newcommand\Z{\text{$\mathbb{Z}$}}

\DeclareMathOperator{\HH}{H}

\DeclareMathOperator{\Aut}{Aut}

\newcommand{\Conf}{\mathcal{C}}
\newcommand{\Real}{\mathcal{R}}

\newcommand{\bgamma}{\bm{\gamma}}

\biboptions{numbers,sort&compress}

\journal{}

\begin{document}

\begin{frontmatter}

\title{Generic combinatorial rigidity of periodic frameworks}

\author[HUJI]{Justin Malestein}
\ead{justinmalestein@math.huji.ac.il}
\author[FUB]{Louis Theran}
\ead{theran@math.fu-berlin.de}
\address[HUJI]{Math Department, Hebrew University, Jerusalem}
\address[FUB]{Institut für Mathematik, Freie Universität Berlin}

\begin{abstract}
We give a combinatorial characterization of generic minimal rigidity for \emph{planar periodic frameworks}.
The characterization is a true analogue of the Maxwell-Laman Theorem from rigidity theory: it is stated
in terms of a finite combinatorial object and the conditions are checkable by polynomial time
combinatorial algorithms.

To prove our rigidity theorem we introduce and develop \emph{periodic direction networks}
and \emph{$\Z^2$-graded-sparse colored graphs}.
\end{abstract}


\end{frontmatter}

\section{Introduction}
A \emph{periodic framework} is an infinite planar structure, periodic with
respect to a lattice representing $\Z^2$,  made of \emph{fixed-length bars}
connected by joints with full rotational degrees of freedom; the allowed continuous motions are
those that preserve the lengths and connectivity of the bars, and the framework's $\Z^2$-symmetry.
A periodic framework is \emph{rigid} if the only allowed motions are Euclidean isometries, and
flexible otherwise.

The forced periodicity is a key feature of this model:
there are structures that are rigid with respect to
periodicity-preserving motions that are flexible if a larger class of motions is allowed.
What is \emph{not} required to be preserved is also noteworthy: the lattice is allowed to change as
the framework moves.

Formally, a periodic framework is given by a triple $(\tilde G, \varphi,\tilde{\bm{\ell}})$ where:
$\tilde{G}$ is a simple infinite graph; $\varphi$ is a free $\mathbb{Z}^2$-action on $\tilde{G}$ by automorphisms
such that the quotient is finite; and $\tilde{\bm{\ell}}=(\tilde{\ell_{ij}})$ assigns a
length to each edge of $\tilde G$.

A \emph{realization} $\tilde G(\vec p,\vec L)$ of a periodic framework $(\tilde G, \varphi,\tilde{\bm{\ell}})$
is defined to be a mapping $\vec p$ of the vertex set $V(\tilde G)$ into $\mathbb{R}^2$ and a representation
$\Z^2 \to \R^2$ encoded by a matrix $\vec L \in \R^{2 \times 2}$ (with $\R^2$ here viewed as translations) such that:
\begin{itemize}
\item the representation is equivariant with respect to the $\Z^2$-actions on $\tilde{G}$ and the plane; i.e.,
$\vec p_{\gamma\cdot i} = \vec p_i + \vec L\cdot \gamma$ for all $i\in V(\tilde{G})$ and
$\gamma\in \Z^2$.
\item The specified edge lengths are preserved by $\vec p$; i.e., $||\vec p_i - \vec p_j||=\tilde{\ell}_{ij}$ for all
edges $ij\in E(\tilde{G})$.
\end{itemize}
The reader should note that together these definitions imply that, to be realizable, an abstract periodic
framework must give the same length to each $\Z^2$-orbit of edges.

A realization
$\tilde G(\vec p,\vec L)$ is \emph{rigid} if the only allowed continuous motions of $\vec p$ and $\vec L$
that preserve the action $\varphi$ and the edge lengths are rigid motions of the plane and
\emph{flexible} otherwise.
If $\tilde G(\vec p,\vec L)$ is rigid but ceases to be so if any
$\mathbb{Z}^2$-orbit of edges in $\tilde G$ is removed it is minimally rigid.
These definitions of periodic frameworks and rigidity are from \cite{BS10}.
(See \secref{continuous} for complete details.)

\subsection{Main theorem}
The topic of this paper is to determine rigidity and flexibility of periodic frameworks based
only on the \emph{combinatorics} of a framework---i.e., which bars are present and not their specific
lengths.  In general, this isn't possible, and even testing rigidity of a \emph{finite} framework seems to
be a hard problem, with the best known algorithms relying on exponential-time Gröbner basis computations.

However, for \emph{generic} periodic frameworks, we give the following
combinatorial characterization, which is analogous to the landmark Maxwell-Laman Theorem \cite{M64,L70}.
The \emph{colored-Laman graphs} appearing in  statements of theorems are defined in \secref{tdngraphs}; the
quotient graph is defined in \secref{graphs}; and
genericity is defined precisely in  \secref{infinitesimal} in terms of the
coordinates of the points in a realization avoiding a nowhere-dense algebraic set.  In particular,
this means that the set non-generic realizations has measure zero.
\begin{theorem}[\periodiclaman]%
\theolab{periodiclaman}
Let $(\tilde G, \varphi,\tilde{\bm{\ell}})$ be a generic periodic framework.  Then a generic
realization $\tilde{G}(\vec p,\vec L)$ of $(\tilde G, \varphi,\tilde{\bm{\ell}})$ is
minimally rigid if and only if its colored quotient graph $(G,{\bm{\gamma}})$ is colored-Laman.
\end{theorem}
\theoref{periodiclaman} is a true combinatorial characterization of generic periodic
rigidity in the plane: $(G,{\bm{\gamma}})$ is a finite combinatorial object and the colored-Laman
condition is checkable in polynomial time. The specialization of \theoref{periodiclaman}
to the case where the quotient graph has only one vertex is implied by \cite[Theorem 3.12]{BS10}.

\subsection{Examples}
Because infinite periodic graphs are unwieldy to work with, we will model periodic frameworks by
\emph{colored graphs}, which are finite directed graphs with elements of $\Z^2$ on the edges.
These are defined in \secref{graphs}, but we show some examples here to give intuition and motivate
\theoref{periodiclaman}.
\begin{figure}[htbp]
\centering
\subfigure[]{\includegraphics[width=.45\textwidth]{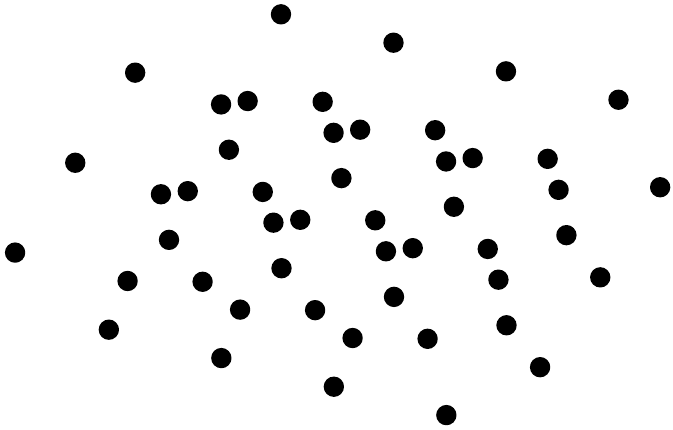}}
\subfigure[]{\includegraphics[width=.45\textwidth]{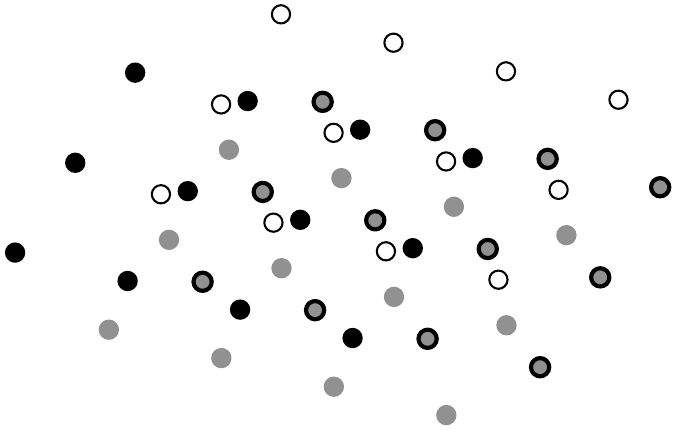}}
\subfigure[]{\includegraphics[width=.45\textwidth]{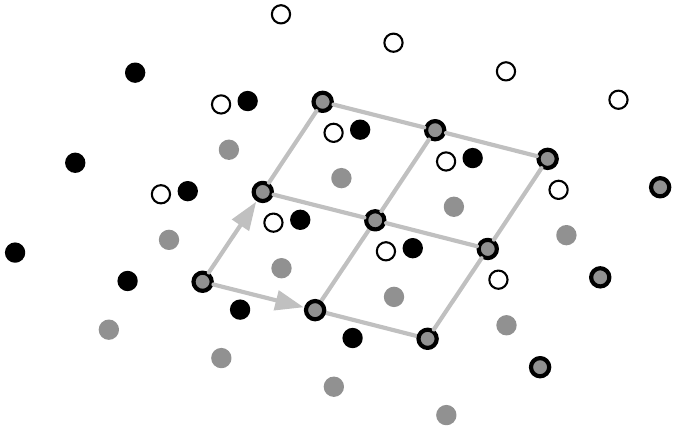}}
\caption{Periodic point sets: (a) part of a periodic point set;
(b) $\Z^2$-orbits of the points indicated by coloring; (c) the representation of $\Z^2$ is indicated by
gray vectors with arrows, copies of the ``unit cell'' shown in gray.}
\label{fig:points-example}
\end{figure}
\figref{points-example} (a) shows part of a periodic point set, (b) makes it more clear that it is indeed periodic by
indicating the $\Z^2$ orbits of the points; (c) indicates the vectors representing $\Z^2$ by translations and
shows several copies of the fundamental domain of the $\Z^2$-action on the plane (often called the unit cell).
\begin{figure}[htbp]
\centering
\subfigure[]{\includegraphics[width=.45\textwidth]{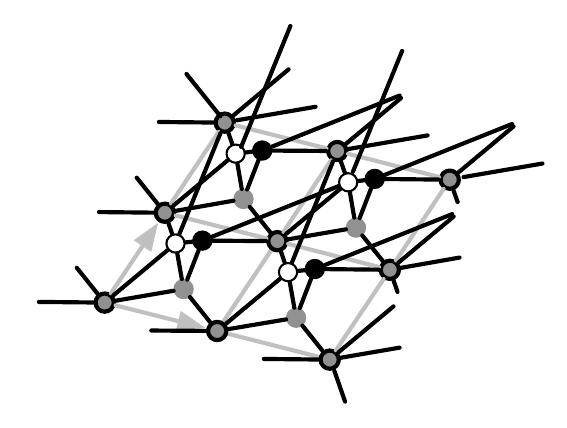}}
\subfigure[]{\includegraphics[width=.45\textwidth]{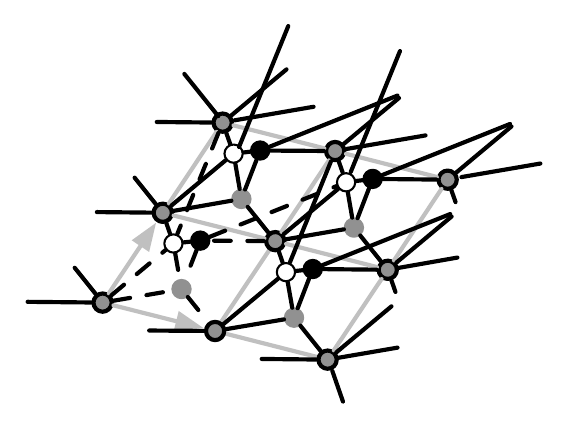}}
\subfigure[]{\includegraphics[width=.45\textwidth]{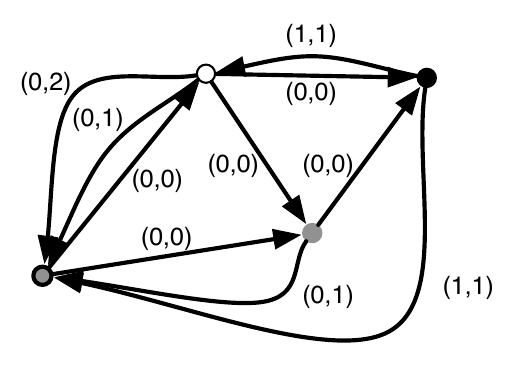}}
\subfigure[]{\includegraphics[width=.45\textwidth]{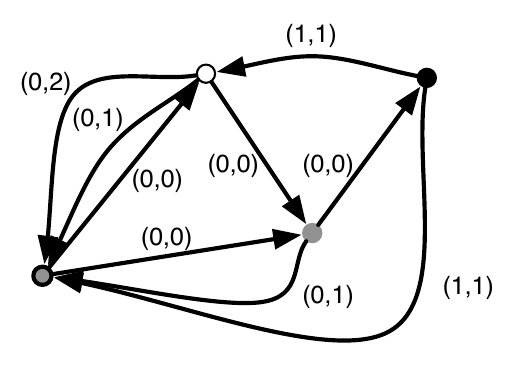}}
\caption{A periodic framework: (a) part of an infinite periodic framework;
(b) choosing representatives for the edges; (c) the associated colored graph; (d) a colored-Laman basis of the
colored graph in (c).}
\label{fig:framework-example}
\end{figure}
\figref{framework-example}~(a) shows an example of a periodic framework, and (b) and (c)
illustrate the construction from moving from a periodic framework to a colored graph, from which
the entire periodic graph formed by the bars of the framework can be reconstructed.   \secref{graphs}
describes the construction in detail.

The framework in \figref{framework-example} turns out to be generically rigid, and, in fact, over-constrained.
This is because the colored graph in \figref{framework-example} (c) has $n=4$ vertices and $m=10$
edges.  On the other hand, there are $12$ total variables defining the framework
($8$ from the coordinates of the points and $4$ from the representation of the lattice) and $3$ trivial degrees of
freedom (from Euclidean motions of the plane), so there can be only $9$ independent distance equations.
Since the graph in \figref{framework-example} (c) is somewhat complicated, seeing that the framework from
\figref{framework-example} (a) is rigid is most easily done via \theoref{periodiclaman}.
\figref{framework-example} (d) indicates a subgraph corresponding to a minimally rigid framework in gray;
this is what we call \emph{colored-Laman}.

\begin{figure}[htbp]
\centering
\subfigure[]{\includegraphics[width=.30\textwidth]{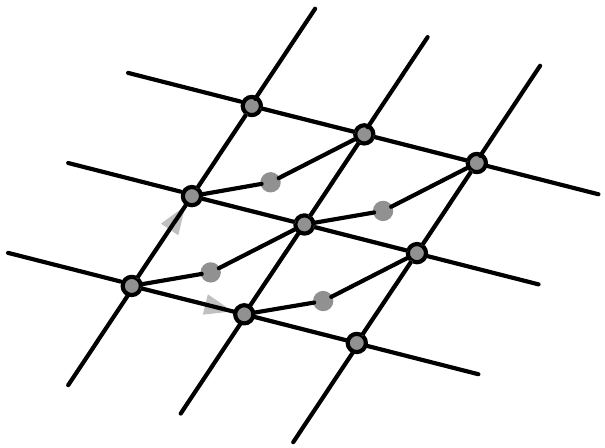}}
\subfigure[]{\includegraphics[width=.30\textwidth]{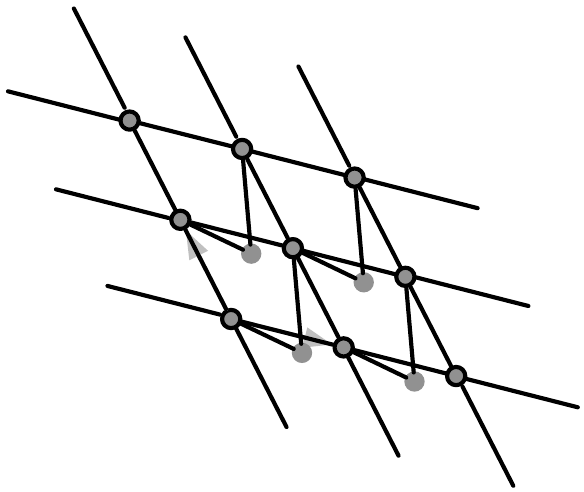}}
\subfigure[]{\includegraphics[width=.30\textwidth]{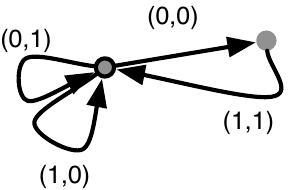}}
\caption{A flexible one degree of freedom periodic framework with a motion that changes the lattice
representation (a) and (b); its associated colored graph
(c).}
\label{fig:flex-framework}
\end{figure}
\figref{flex-framework} example shows an example of a very simple one degree of freedom framework (a) and its
associated colored graph (c).  Two things to note are: its flex (b) necessarily moves the representation of the
lattice; in this rigidity model \emph{all} $4$-regular colored graphs are associated with frameworks that,
generically, have at least one degree of freedom.

\begin{figure}[htbp]
\centering
\subfigure[]{\includegraphics[width=.45\textwidth]{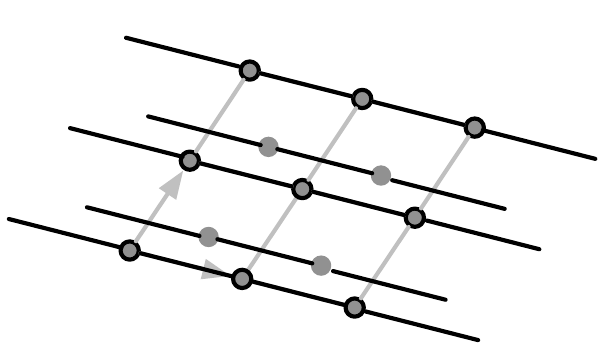}}
\subfigure[]{\includegraphics[width=.45\textwidth]{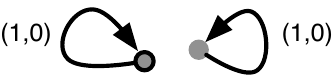}}
\caption{A dependent, but disconnected, periodic framework (a) and its associated colored graph
(b).}
\label{fig:dependent-framework}
\end{figure}
Our final example, in \figref{dependent-framework} is a simple illustration of a subtle
feature of the periodic rigidity model that the colored-Laman graphs we define capture: dependencies
between disconnected sub-frameworks and sub-graphs.  The framework in \figref{dependent-framework} (a)
consists of two disconnected orbits of bars, each of which is modeled by a self-loop in the associated
colored graph shown in \figref{dependent-framework} (b).  To be realizable at all, the
lengths of the bars all need to be exactly the same: in both cases, what is being restricted is just the length
of one of the lattice vectors.  It follows that, generically, this framework is over-constrained,
despite both it and its quotient being disconnected.
Accounting for the interactions between sub-frameworks that affect the
same parts of the lattice representation is a key aspect of the definition of colored-Laman graphs.

\subsection{Periodic direction networks}
A \emph{periodic direction network}
$(\tilde G, \varphi,\tilde{\vec d})$ is an infinite multigraph $\tilde G$ with a
free $\mathbb{Z}^2$-action $\varphi$ by authormorphisms and an assignment of directions
$\tilde{\vec d} = (\tilde{\vec d}_{ij})_{ij\in E(\tilde G)}$ to the edges of $\tilde G$.

A \emph{realization} $\tilde G(\vec p,\vec L)$
of a periodic direction network is a mapping $\vec p$ of the vertex set $V(\tilde G)$ into
$\mathbb{R}^2$ and a matrix $\vec L\in \R^{2\times 2}$ representing
$\mathbb{Z}^2$ by translations of $\mathbb{R}^2$ such that:
\begin{itemize}
\item The representation $\Z^2 \to \R^2$ from $\vec L$ is equivariant with respect to the actions on $\tilde{G}$ and the plane;
i.e., $\vec p_{\gamma\cdot i} = \vec p_i + \vec L\cdot \gamma$ for all $i\in V(\tilde{G})$ and
$\gamma\in \Z^2$.
\item The specified edge directions are preserved by $\vec p$; i.e.,
$\vec p_j-\vec p_i=\alpha_{ij}\tilde{\vec d}_{ij}$ for all edges $ij\in E(\tilde G)$ and some
$\alpha_{ij}\in \mathbb{R}$
\end{itemize}
An edge $ij$ is \emph{collapsed} in a realization
$\tilde G(\vec p)$ if $\vec p_i = \vec p_j$; a realization in which all edges are collapsed is defined to be
a \emph{collapsed realization}, and a realization in which no edges are collapsed is \emph{faithful}.  We prove an
analogue of Whiteley's Parallel Redrawing Theorem \cite{W96,W88}.
\begin{theorem}[\periodicparallel]
\theolab{periodicparallel}
Let $(\tilde G,\varphi,\tilde{\vec d})$ be a generic periodic direction network.
Then $(\tilde G,\varphi,\tilde{\vec d})$ has a unique, up to translation and scaling,
faithful realization if and only if its quotient graph $(G,{\bm{\gamma}})$ is colored-Laman.
\end{theorem}

\subsection{Roadmap and guide to reading}
Let us first briefly sketch how
\theoref{periodicparallel} implies \theoref{periodiclaman}.
All known proofs of ``Maxwall-Laman-type'' theorems (such as \theoref{periodiclaman}) proceed via
a linearization of the problem called \emph{infinitesimal rigidity}, which is
concerned with the rank of the  \emph{rigidity matrix} $\vec M_{2,3,2}(G,{\bm{\gamma}})$.  The structure
of the rigidity matrix depends on the quotient graph $(G,{\bm{\gamma}})$ of the periodic framework.
The two main steps are to prove that for a quotient graph with $n$ vertices and $m$ edges:
\begin{itemize}
\item If $\vec M_{2,3,2}(G,{\bm{\gamma}})$ has rank $2n+1$ then the associated framework is rigid.
(This is done in  \cite{BS10}.)
\item For almost all frameworks (these are called \emph{generic}) with quotient graph $(G,{\bm{\gamma}})$
having $n$ vertices and $2n+1$ edges, the rank of $\vec M_{2,3,2}(G,{\bm{\gamma}})$ is $2n+1$ if and only
if $(G,{\bm{\gamma}})$ is colored-Laman.
\end{itemize}
The second step, where the rank of the rigidity matrix is proved from only a combinatorial assumption is
the (more difficult) ``Laman direction,'' and it occupies most of this paper.
The approach is as follows:
\begin{itemize}
\item We begin with a matrix $\vec M_{2,2,2}(G,{\bm{\gamma}})$, arising from a periodic direction network
that has non-zero entries in same positions as the rigidity matrix, but simpler entries
(vectors $(a_{ij},b_{ij})$  instead  of differences of points $\vec p_i-\vec p_j$).
The rank of $\vec M_{2,2,2}(G,{\bm{\gamma}})$	is much easier to analyze directly.
\item Then we apply \theoref{periodicparallel} to a colored-Laman graph $(G,{\bm{\gamma}})$.  For
generic directions $\vec d$ (defined in \secref{directions}), there is a point set $\vec p$ and a
lattice $\vec L$ such that $\vec p_i-\vec p_j$ is in the direction of $\vec d_{ij}$ and
the two endpoints of every edge are different.  Substituting in the $\vec p_i$ recovers
the rigidity matrix $\vec M_{2,3,2}(G,{\bm{\gamma}})$ from $\vec M_{2,2,2}(G,{\bm{\gamma}})$,
completing the proof.
\end{itemize}

\noindent The main task, then, is to prove \theoref{periodicparallel}.  This proceeds in three distinct steps as follows.

\noindent \textbf{The combinatorial step:} We begin by defining periodic and colored
graphs (\secref{graphs}). Sections \ref{section:colored-laman}--\ref{section:cloning} define and develop
our central combinatorial objects: \emph{colored-Laman} and \emph{$(2,2,2)$-colored-graphs}.
Our main combinatorial results are that:
\begin{itemize}
\item $(2,2,2)$-colored-graphs give the bases of a matroid (\lemref{222matroid}).
\item $(2,2,2)$-colored graphs have an alternate characterization via a	sparsity-certifying
decomposition into edge-disjoint subgraphs 	(\lemref{222decomp}).
\item Colored-Laman graphs are related to $(2,2,2)$-colored-graphs
by doubling an edge (\lemref{periodiclamancontract}).
\end{itemize}

\noindent \textbf{The natural representation step:}
In Sections \ref{section:natural}--\ref{section:22k-rep} we prove that the matroid
that has $(2,2,2)$-colored graphs as its bases admits a \emph{natural representation}:
i.e., for a periodic graph $(G,{\bm{\gamma}})$ with $n$ vertices and $2n+2$ edges,
there is a matrix $\vec M_{2,2,2}(G,{\bm{\gamma}})$ that has rank $2n+2$
if and only if $(G,{\bm{\gamma}})$ is $(2,2,2)$-colored.  The decomposition
\lemref{222decomp} lets us study $\vec M_{2,2,2}(G,{\bm{\gamma}})$ in terms
of highly structured minors for which it is easier to give exact determinant
formulas.

\vspace{.75 ex}
\noindent \textbf{The geometric step:} Sections \ref{section:directions}--\ref{section:parallel} develop
the geometric theory of
periodic direction networks and gives the proof of \theoref{periodicparallel}.
The equations giving the edge directions are closely related to the natural
representation matrix $\vec M_{2,2,2}(G,{\bm{\gamma}})$, which allows us to
combinatorially	predict the generic rank.  To understand the geometry of a generic
direction network's realization space, we make a connection between collapsed edges
in a realization and doubling of edges, bringing us back, in a natural way to
colored-Laman graphs.

\subsection{Related work: generic periodic rigidity}
The question of generic rigidity and flexibility of periodic frameworks has been studied
in the past.  The first result of which we are aware is by Walter Whiteley \cite{W88},
who showed (in our language) that there are colored graphs which lead to infinitesimally
rigid periodic frameworks if the representation of the lattice is fixed. Elissa
Ross \cite{R09,R11} proved, using methods quite different from ours, the analogue of
\theoref{periodiclaman} when the lattice representation is fixed.

\subsection{Related work: parallel redrawing}
Our proof of \theoref{periodicparallel} is essentially algebraic, but the geometric
correspondence between \emph{parallel redrawings} of finite frameworks (which are vertex
displacements preserving the edge directions) and \emph{infinitesimal motions} (which
preserve the lengths of the bars to first order, see \secref{infinitesimal})
has a long history as a folklore tool in engineering.  The reason for this is
that it is easier to see a non-trivial parallel redrawing than the associated infinitesimal motion.
Whiteley developed the subject \cite[Section 4]{W96} and its generalizations to higher-dimensional
\emph{scene analysis} \cite[Section 8]{W96}, and the idea of deducing the
Maxwell-Laman Theorem from a direct proof of the Parallel Redrawing Theorem appears
in \cite{W88}.

\subsection{Related work: $(k,\ell)$-sparse graphs}
The \emph{colored-Laman} and \emph{colored-$(2,2,2)$ graphs} we define in
\secref{tdngraphs} are an extension of the well-studied families of
\emph{$(k,\ell)$-sparse} \cite{LS08} graphs: these are defined by the
property that any subgraph spanning $n'$ vertices and $m'$ edges satisfies
$m'\le kn'-\ell$.  All known generic rigidity characterizations (e.g., \cite{L70,T84,ST10,W88}) are
in terms of $(k,\ell)$-sparse graphs with specialized parameters for $k$ and $\ell$.
The combinatorial theory we develop in \secref{tdngraphs} is a generalization of
$(k,\ell)$-sparsity, and runs in parallel to parts of it:
\begin{itemize}
\item \lemref{periodiclamancontract}, which relates colored-Laman graphs to colored-$(2,2,2)$
graphs by edge doubling is a colored graph version of theorems of Lovász-Yemini \cite{LY82} and
Recski \cite{R84}.
\item The decomposition \lemref{222decomp} is analogous to the Nash-Williams-Tutte Theorem
\cite{NW61,T61}, and we use it in a similar manner to the way \cite{W88} and \cite{ST10}
use Nash-Williams-Tutte.
\end{itemize}

\subsection{Related work: matroid constructions}
Another way to view our $\Z^2$-graded-sparse graphs is in terms of constructions on
matroids. To make the connection clear, we give an outline of our proof in these terms:
\begin{itemize}
\item We argue directly that $(1,1,2)$-graphs are the bases of a matroid with a specific
rank function (\lemref{fmat}).
\item That $(2,2,2)$-graphs give the bases of a matroid (\lemref{222matroid})  and the
associated decomposition follows from the Nash-Williams matroid construction.
\item The colored-Laman matroid can then be obtained via Dilworth Truncation \cite[Section 7.7]{B86},
which, in terms of submodular functions means going from $f(\cdot)$ to $f(\cdot) - 1$
as the function generating the matroid.
\end{itemize}
\noindent For linearly representable matroids, Matroid Union and Dilworth Truncation both have
counterparts in terms of the representation.  These are:
\begin{itemize}
\item Adding a new block to the representing matrix with the
same filling pattern but different generic variables, for Matroid Union
(\cite[Proposition 7.16.4]{B86}).
\item Confining the rows of the representing matrix to
lie in a generic hyperplane, for Dilworth Truncation (see \cite{L77}).
\end{itemize}
The subtlety, for us, is that we are after a \emph{specific}
representation for the colored-Laman graph matroid, namely the rigidity matrix $\vec M_{2,3,2}(G)$.
This means that we need an argument specific to our geometric setting, and not just
general results on Dilworth Truncation.

\subsection{Notations}
In what follows $(\tilde G,\varphi)$ denotes an (infinite) periodic graph with $\mathbb{Z}^2$-action $\varphi$.
Colored graphs are denoted $(G,{\bm{\gamma}})$.  $G$ always refers to a finite graph with vertex
set $V=V(G)$ and edge set $E(G)$.  For finite graphs, the parameters $n$, $m$ and $c$ refer to the
number of vertices, edges, and connected components respectively.
Graphs may have multiple edges (and so are multi-graphs) as well as self-loops;
edges and loops are oriented and denoted $ij$, where $i$ is the tail and $j$ is the head.
Since edges and self-loops play the same role in the combinatorial and linear theories appearing
here, for economy of language, we do not distinguish between them.  Self-loops are
treated as edges with the same vertex playing the role of head and tail; self-loops are
also treated as cycles with only one vertex. Multiple copies of the same edge are treated as distinguished;
where there is a source of confusion in the indexing, we explicitly note
it.  The coloring ${\bm{\gamma}}=(\gamma_{ij})_{ij\in E}$ is a vector mapping edges $ij$ to elements $\gamma_{ij}=(\gamma^1_{ij},\gamma^2_{ij})$
of $\mathbb{Z}^2$.

Subgraphs $(G',{\bm{\gamma}})$ of a colored graph $(G,{\bm{\gamma}})$ are taken to be edge-induced, with the induced
coloring from ${\bm{\gamma}}$, and have $n'$ vertices, $m'$ edges, and $c'$ connected components.

A matrix $\vec M$ is denoted in bold; if $\vec M$ is $m\times n$, then $\vec M[A,B]$ denotes
the submatrix induced by the row indices $A\subset [m]$ and column indices $B\subset [n]$.  Vectors
$\vec v$ are denoted in bold.  Point sets $\vec p=(\vec p_i)_1^n\subset \mathbb{R}^2$ of
$n$ points in the plane are taken both as indexed sets of points and flattened vectors $\vec p\in \mathbb{R}^{2n}$;
each point $\vec p_i=(x_i,y_i)$.

\subsection{Acknowledgements}
We would like to thank Igor Rivin for his encouragement and many productive discussions on this topic.
The anonymous referee's comments helped us greatly improve our exposition, and we thank them for their
careful reading. We became interested in this problem as part of a larger project to study the
rigidity and flexibility of zeolites \cite{SWTT06,R06,T04},
which is supported by  CDI-I grant DMR 0835586 to Rivin and M. M. J. Treacy.  LT's final
preparation of this paper was supported by the European Research Council under the European
Union's Seventh Framework Programme (FP7/2007-2013) / ERC grant agreement no 247029-SDModels.

\section{Colored and periodic graphs}\seclab{graphs}
In this section we introduce \emph{periodic graphs} and \emph{colored graphs}, which provide the
combinatorial setting for this paper.  What we call colored graphs are also known as
``gain graphs'' or ``voltage graphs'' \cite{Z98}; the terminology of colored graphs
originates with Igor Rivin's work on hypothetical zeolites \cite{R06}.

\subsection{Periodic graphs}
A {\em periodic graph} $(\tilde{G}, \varphi)$ is defined to be a simple infinite graph with an associated action
$\varphi: \Z^2 \to \Aut(\tilde{G})$ that is free on edges and vertices and such that the quotient graph
$\tilde{G}/\Z^2$ is finite.  We denote the action of $\gamma \in \Z^2$ on a vertex $i$ or edge
$ij$ by $\gamma \cdot i$ and $\gamma \cdot (ij)$ respectively.

\subsection{$\Z^2$-colored graphs}
We now define $\Z^2$-colored graphs (shortly, \emph{colored graphs}, since all the
graphs in this paper have $\Z^2$-colors), which are finite objects that capture the
essential information in a periodic graph.  A {\em $\Z^2$-colored graph}
$(G, {\bm{\gamma}})$ is a directed graph $G$ together with a vector
${\bm{\gamma}}_{ij}=(\gamma_{ij})$ assigning each edge of $G$ a group element in
$\mathbb{Z}^2$, which we call the \emph{color}.

Given a colored graph $(G,\bgamma)$, we define its \emph{development} $(\tilde{G},\varphi)$
as follows:
set
\[V(\tilde{G}) =  \{ (i, \gamma) \;\; | \;\; i \in V(G) \;\; \gamma \in \Z^2\}\]
Similarly, define  the edge set as $E(\tilde{G}) = E(G) \times \Z^2$.  Set the tail of
$(ij, \gamma)$ to be  $(i, \gamma)$ and the head of $(ij, \gamma)$ to be $(j, \gamma_{ij} +\gamma)$.
These definitions induce a $\Z^2$-action $\phi$ via
\begin{eqnarray*}
\gamma' \cdot (i,\gamma) = (i, \gamma' +\gamma) & \text{and} & \gamma' (ij, \gamma) = (ij, \gamma' +\gamma)
\end{eqnarray*}

We can also define a colored graph, called the \emph{colored quotient},
for each periodic graph  $(\tilde{G}, \varphi)$. Let the quotient graph
$G = \tilde{G}/\Z^2$ have $n$ vertices and $m$ edges.  Select one representative from each
$\Z^2$-orbit of vertices in $V(\tilde G)$ (there will be $n$ of these) to represent $V(G)$.
We define a coloring ${\bm{\gamma}}$ of $G$ as follows:
\begin{itemize}
\item Orient $G$ in some way.
\item Let $ij\in E(G)$ be a directed edge of $G$ and let $\tilde {ij}$ be the unique lift
of $ij$ with tail $\tilde i$ being the chosen representative for $i\in V(G)$.
\item The head of $\tilde {ij}$ is $\gamma\cdot\tilde j$ where $\tilde j$
is the chosen representative of $j$ in $V(\tilde G)$ and $\gamma$ is uniquely defined by
the choices of $\tilde i$ and $\tilde j$.  We define this $\gamma$ to be
the color $\gamma_{ij}$ on $ij\in E(G)$.
\end{itemize}
If $(\tilde{G}, \varphi)$ is already realized on a periodic point set, a convenient way to
pick the representatives is by the points in a fundamental domain of the $\Z^2$-action on the plane,
as we did in \figref{framework-example} (b).
The following lemma shows that the associated colored graph encodes all
the data of the periodic graph.
\begin{lemma}\lemlab{dictionary}
The development of every colored graph is a periodic graph, and, in particular, every
periodic graph is the development of its colored quotient, for any choice of representatives.
\end{lemma}
\begin{proof}
Let $(G,\bgamma)$ be a colored graph.  It is clear from the definition
that the $\Z^2$-action $\phi$ on the development $\tilde{G}$ acts
freely one the edges and vertices, so $(\tilde{G},\phi)$ is periodic.
Moreover, by construction, $G$, as an undirected graph, is the quotient $\tilde{G}/\Z^2$.
To see that $(G,\bgamma)$ is the colored quotient of $\tilde{G}$, let $G$ retain its
original orientation and select as representatives for $V(G)$ the vertices $(i,(0,0))\in V(\tilde{G})$,
and for edges $ij\in E(G)$ the unique lift of $ij$ with tail $(i,(0,0))$ this lift then has
head $(j,\gamma_{ij})$, as desired.
\end{proof}
In slightly different language, \lemref{dictionary} says that the colors in a colored
graph $(G,\bgamma)$ encode the covering map to $G$ from its development $\tilde{G}$
induced by $\varphi$.

\subsection{The $\mathbb{Z}^2$-rank  of a colored graph}
Different choices of vertex representatives of $V(G)$ will yield different colorings.
The ``right'' finite graph counterpart to $(\tilde{G}, \varphi)$ is the finite
graph $G = \tilde{G}/\Z^2$ along with the data of a homomorphism
$\rho: \HH_1(G, \Z) \to \Z^2$, where $\HH_1$ is the first homology group.
One can show that there is a natural bijective correspondence between pairs $(\tilde{G}, \varphi)$
and pairs $(G, \rho)$.  For our purposes, it will be technically simpler to use colored graphs.

A fundamental notion used in our paper is that of the $\Z^2$-rank of a
colored graph. Let $C$ be a simple cycle in a colored graph $(G,{\bm{\gamma}})$ and fix a
traversal order of $C$. We define the function $\rho(C)$ to be:
\[
\rho(C) = \left(\sum_{\substack{\text{$ij\in C$ }\\ \text{traversed from $i$ to $j$}}}\gamma_{ij}\right) -
\left(\sum_{\substack{\text{$ij\in C$} \\\text{traversed from $j$ to $i$}}}\gamma_{ij}\right)
\]
With a slight abuse of notation, we denote the induced homomorphism $\HH_1(G, \Z) \to \Z^2$ by $\rho$
as well.  The \emph{$\Z^2$-rank of $(G, \bm{\gamma})$} is then defined to be the rank of the
subgroup of $\Z^2$ generated by the image $\rho(G, \bm{\gamma})$ of $\rho$.
Equivalently, this is the number of linearly independent vectors in $\rho(G, \bm{\gamma})$.

Note that the $\mathbb{Z}^2$-rank is invariant under the choice of representatives
when taking the quotient of a periodic graph to obtain a colored graph.

\subsection{Facts about the $\Z^2$-rank}
The following basic lemmas about the $\Z^2$-rank will be useful later.
\begin{lemma}\lemlab{z2-basis}
Let $(G,\bm{\gamma})$ be a colored graph with $\mathbb{Z}^2$-rank $k$.  Then for any choice
of cycle basis $B$ for $\HH_1(G,\Z)$, there are $k$ cycles in $B$ with independent images
in $\mathbb{Z}^2$.
\end{lemma}
\begin{proof}
If the map $\rho$ extends to a well-defined linear map on $\HH_1(G,\Z)$, then we are done.
This follows from the fact that if the cycle $C_3=C_1\Delta C_2$, then if edges in $C_1\cap C_3$
are traversed forwards, edges in $C_2\cap C_3$ are traversed backwards. Thus
$\rho(C_3)=\rho(C_1)+\rho(C_2)$, since the contributions of edges in $C_1\cap C_2$ cancel
on the r.h.s.
\end{proof}

\begin{lemma}\lemlab{z2-rank-01}
Let $(G,\bgamma)$ be a connected colored graph, and let $ij$ be colored edge not in $E(G)$.
Then the $\Z^2$-rank of $(G+ij,\bgamma)$ is at most one more than
the $\Z^2$-rank of $(G,\bgamma)$.
\end{lemma}
\begin{proof}
Pick a cycle basis for $\HH_1(G,\Z)$ and then extend it to a basis of $\HH_1(G+ij,\Z)$.  This
process adds at most one cycle $C$, and, thus, the dimension of $\rho(G+ij,\bgamma)$ is
at most one more than that of $\rho(G,\bgamma)$.
\end{proof}

\begin{lemma}\lemlab{z2-rank}
Let $(G,\bm{\gamma})$ be a colored graph and let $(G',\bm{\gamma})$ be a subgraph of $(G,\bm{\gamma})$.
Suppose that adding the (colored) edge $ij$ causes the $\Z^2$-rank of $(G,\bm{\gamma})$ to
increase, and that $ij$ is spanned by a connected component of $G'$.  Then adding
$ij$ to $G'$ causes its $\Z^2$-rank to increase.
\end{lemma}
\begin{proof}
Pick a spanning forest $F'$ in $G'$ and extend it to a spanning forest $F$ of $G$.  The collection of
cycles formed by each edge of $E(G)-F$ and $F$ give a basis of $\HH_1(G, \Z)$.  In particular,
since adding the edge $ij$ to $G$ causes an increase in the $\Z^2$-rank, \lemref{z2-basis}
implies that the $\rho$-image of the cycle formed by $ij$
and $F$ is not in the span of $\rho(G,\bm{\gamma})$.  By hypothesis, this cycle uses only $ij$ and
edges from $F'$, and since $\rho(G',\bm{\gamma})$ is contained in $\rho(G,\bm{\gamma})$,
the $\Z^2$-rank of $(G',\bm{\gamma})$ increases as well.
\end{proof}

\subsection{The $\Z^2$-image and the development}
Although none of our proofs explicitly depend on it, we will want to refer to the following
fact about how the $\Z^2$-image of a colored graph relates to the connectivity of its development.
In the interest of space, we skip the proof, which is straightforward, but requires a case analysis.
\begin{lemma}\lemlab{development-connectivity}
Let $(G,\bgamma)$ be a colored graph with $G$ connected and $\Z^2$-rank k.  Let $\Gamma(G) < \Z^2$ be the subgroup
of $\Z^2$ generated by $\rho(G,\bgamma)$, and let $\ell$ be the index of $\Gamma(G)$ in $\Z^2$.  Let $(\tilde{G},\varphi)$
be the development. Then:
\begin{itemize}
\item If $k=2$, then $\ell$ is finite, and the development has $\ell$ infinite connected components.
\item If $k=1$, then $\ell$ is infinite, and the development has infinitely many
infinite connected components.  These components can be indexed by $\Z\times [\ell']$, where
$\ell'$ is the index of a subgroup of $\Z$.
\item If $k=0$, then $\ell$ is infinite and the development has infinitely many finite connected components,
which can be indexed by $\Z\times \Z$.
\end{itemize}
\end{lemma}
\figref{dependent-framework} (a) shows an example of the $k=1$ case of \lemref{development-connectivity}.

\subsection{Doubling edges}
In the sequel, we make use of the operation of \emph{doubling an edge $ij$} in colored graphs.
Let $(G,{\bm{\gamma}})$ be a colored graph, and let $ij\in E(G)$ be an edge with color $\gamma_{ij}$.
Then, the {\em graph with $ij$ doubled, $(G+(ij)_c, \bm{\gamma})$} is the colored graph
$(G,{\bm{\gamma}})$ with a new edge $(ij)_c$ with the same color $\gamma_{ij}$ as $ij$.

\section{Colored-Laman graphs}\seclab{colored-laman}
Let $(G,{\bm{\gamma}})$ be a colored graph.  We define $(G,{\bm{\gamma}})$ to be
\emph{$(2,3,2)$-$\mathbb{Z}^2$-graded-sparse} (shortly, \emph{colored-Laman-sparse})
if for all edge-induced subgraphs $(G',{\bm{\gamma}}')$
of $(G,{\bm{\gamma}})$ on $n'$ vertices, $m'$ edges, $c'$ connected components,
and $\Z^2$-rank $k'$,
\[
m'\le 2n' - 3 + 2k' - 2(c'-1)
\]
If, in addition, $(G,{\bm{\gamma}})$ has $m = 2n+1$ edges, we define $(G,{\bm{\gamma}})$ to be a
\emph{colored-Laman graph}.

Because of the key role they play, the rest of this section gives examples of colored graphs
that are and are not colored-Laman, with some discussion of the aspects of periodic rigidity
each example captures.  \secref{cloning} proves the key properties of colored-Laman graphs that we need;
these are stated in terms of the related family of $(2,2,2)$-graphs, which are developed in
\secref{tdngraphs}.

\subsection{Connection between colored-Laman graphs and Laman graphs}
If the rank in $\mathbb{Z}^2$ is zero for all subgraphs, then
for a graph to be colored-Laman-sparse, we have
\[
m' \le 2n' - 3 - 2(c'-1)
\]
The Laman graphs that characterize minimal rigidity of finite planar frameworks are defined by the
counting condition ``$m'\le 2n' - 3$'', but it is not hard to show, using \cite[Lemma 4]{LS08},
that minimal violations of (uncolored) Laman-sparsity happen on connected subgraphs,
so the two conditions give the same family of graphs.  Thus Laman-sparse graphs, with any
coloring, are always colored-Laman-sparse.

\subsection{The smallest colored-Laman graph}
\figref{colored-laman-one-vertex} shows an example of the smallest colored-Laman graphs:
a single vertex with three self-loops with colors such that every two colors
are linearly independent.  Rigidity of this example, and its higher-dimensional generalizations was
shown in \cite{BS10}. Note that if any two of the colors are linearly dependent,
the subgraph containing just these loops is not colored-Laman sparse.
\begin{figure}[htbp]
\centering
\includegraphics[width=.3\textwidth]{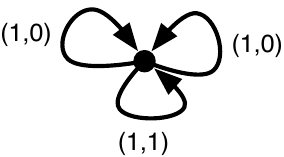}
\caption{A colored-Laman graph with one vertex.}
\label{fig:colored-laman-one-vertex}
\end{figure}

\subsection{A larger example}
\figref{colored-laman-larger} is a larger example of a colored-Laman graph.
\begin{figure}[htbp]
\centering
\includegraphics[width=.5\textwidth]{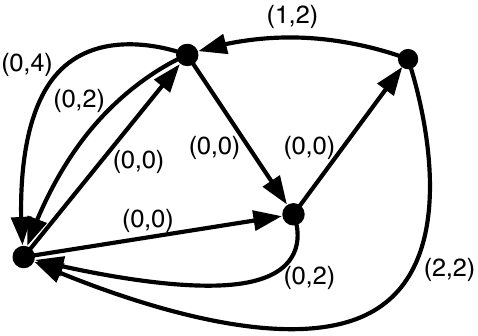}
\caption{An example of a colored-Laman graph.}
\label{fig:colored-laman-larger}
\end{figure}

\subsection{Colored-Laman graphs need not be Laman-spanning}
\figref{colored-laman-not-laman-spanning} (c) shows that colored-Laman graphs
need not have any spanning subgraph that, after forgetting the colors,
corresponds to a minimally rigid generic finite framework.
\begin{figure}[htbp]
\centering
\includegraphics[width=.55\textwidth]{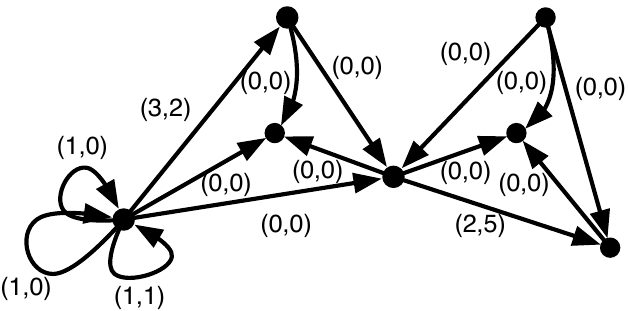}
\caption{A colored-Laman graph that does not have a Laman graph as a spanning subgraph.}
\label{fig:colored-laman-not-laman-spanning}
\end{figure}

\subsection{Example: checking edge-induced subgraphs is essential}
That the colored-Laman counts are defined on subsets of
\emph{edges} and not subsets of vertices is essential: the colored graph in \figref{connected-example} has $2n+1$
edges, and meets the colored-Laman counts on all vertex-induced subgraphs, but it is \emph{not} colored-Laman.
The subgraph in \figref{connected-example} that fails the colored-Laman condition is the
colored graph  \figref{dependent-framework} (b).
\begin{figure}[htbp]
\centering
\includegraphics[width=0.6\textwidth]{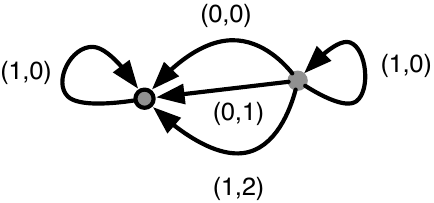}
\caption{A colored graph that is \emph{not} colored-Laman, but meets the sparsity
counts on all \emph{vertex induced}
subgraphs.  Edges are shown with their orientation and colors.
Sparsity is violated on the disconnected subgraph including the two self-loops
with the same color.}
\label{fig:connected-example}
\end{figure}

\subsection{Finite-index $\Z^2$-image}\seclab{finite-index}
\figref{colored-laman-finite-index} (a) shows a slight variation of the example in \figref{colored-laman-one-vertex}:
the underlying colored graph is still just three self-loops, but the second coordinates of all the colors
have been multiplied by two.
\begin{figure}[htbp]
\centering
\subfigure[]{\includegraphics[width=0.3\textwidth]{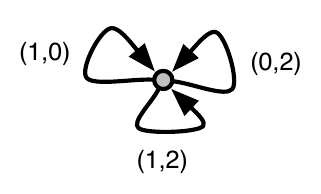}}
\subfigure[]{\includegraphics[width=0.5\textwidth]{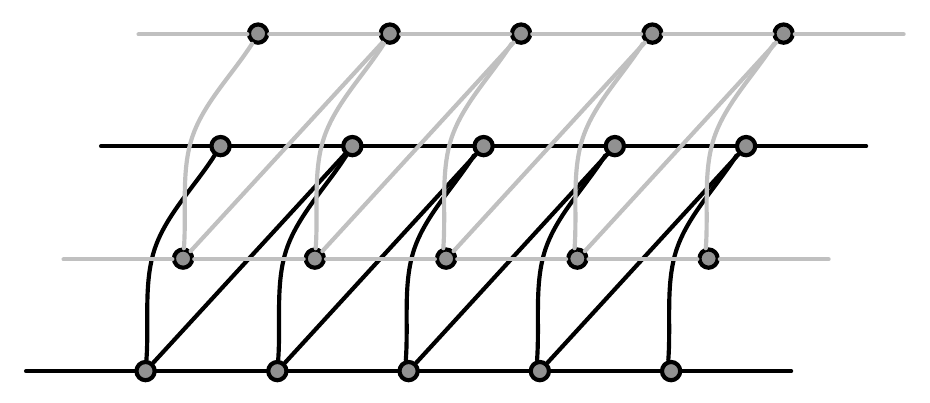}}
\caption{A colored-Laman graph with $\Z^2$-image generating a finite-index subgroup of $\Z^2$ (a)
and its development (b) with connected components indicated by color.}
\label{fig:colored-laman-finite-index}
\end{figure}\lemref{development-connectivity} implies that developing this colored graph to a
periodic graph, as in \figref{colored-laman-finite-index} (shown just as a periodic graph, since the bars of
an associated framework would, necessarily overlap), yields two connected components.
This example illustrates that forcing all the motions to preserve periodicity with
respect to the $\Z^2$-action means that there are rigid periodic frameworks that are not connected,
even though their quotient graph is.

\section{$(2,2,2)$-colored graphs}\seclab{tdngraphs}
Let $(G,{\bm{\gamma}})$ be a colored graph with $n$ vertices and $m$ edges.
We define $(G,{\bm{\gamma}})$ to be \emph{$(2,2,2)$-$\mathbb{Z}^2$-graded-sparse} (shortly, \emph{$(2,2,2)$-sparse})
if for all edge-induced subgraphs $(G',{\bm{\gamma}}')$ of $(G,{\bm{\gamma}})$ on $n'$ vertices, $m'$ edges,
$c'$ connected components, and $\Z^2$-rank $k'$,
\[
m'\le 2n' - 2 + 2k' - 2(c'-1)
\]
If, in addition, $(G,{\bm{\gamma}})$ has $m = 2n+2$ edges, we define $(G,{\bm{\gamma}})$ to be a
\emph{$(2,2,2)$-colored-graph} (shortly, a \emph{$(2,2,2)$-graph}).  More generally,
if $(G,\bgamma)$ is $(2,2,2)$-sparse, has $\Z^2$-rank $k$, and
has $m=2n-2 + 2k$ edges,
we define it to be a \emph{$(2,2,k)$-colored-graph}, or a \emph{$(2,2,k)$-graph}, for short.

\subsection{Relationship to $(k,\ell)$-sparsity}
Our definition of colored-Laman-sparsity and $(2,2,2)$-sparsity has two features that
distinguish it from the traditional $(k,\ell)$-sparsity counts:
\begin{itemize}
\item The number of edges allowed in a subgraph is controlled by the $\Z^2$-rank.
\item The number of edges in a colored-Laman graph on $n$ vertices is $2n+1$;
similarly, a $(2,2,2)$-graph has $2n+2$ edges.  This is outside of
the ``matroidal range'' \cite[Theorem 2]{LS08} for $(k,\ell)$-sparse graphs.
\end{itemize}

The next several sections develop the combinatorial and matroidal properties of $(2,2,k)$-graphs that
we need to prove \theoref{periodicparallel}.  We start with some preliminaries from matroid
theory.

\subsection{Matroid preliminaries}
A \emph{matroid} $\mathcal{M}$ on a ground set $E$ is a combinatorial structure
that captures properties of linear independence.  Matroids have many
equivalent definitions, which may be found in a monograph such as \cite{O06}.  For our
purposes, the most convenient formulation is in terms of \emph{bases}: a matroid $\mathcal{M}$
on a finite ground set $E$ is presented by its bases $\mathcal{B}\subset 2^E$, which
satisfy the following properties:
\begin{itemize}
\item The set of bases $\mathcal{B}$ is not empty.
\item All elements $B\in \mathcal{B}$ have the same cardinality, which is the \emph{rank}
of $\mathcal{M}$.
\item For any two distinct bases $B_1,B_2\in \mathcal{B}$, there are elements
$e_1\in B_1-B_2$ and $e_2\in B_2$ such that $B_2+ \{ e_1 \}-\{ e_2\}\in \mathcal{B}$.
\end{itemize}
In this paper, the ground set is the colored graph $(K_{n}^{6,4},{\bm{\gamma}})$, where $K_{n}^{6,4}$ is the complete
graph on $n$ vertices with $6$ distinguished copies of each edge and $4$ distinguished self-loops on
each vertex.  The coloring ${\bm{\gamma}}$ can vary, but will always be fixed in the statements of theorems; i.e.,
we are defining a family of matroids indexed by coloring.

The reason we define the ground set this way is because at most $6$ parallel edges or $4$ self-loops can appear in
any $(2,2,2)$-sparse colored graph.  Thus there is no loss of generality in making this restriction.

In addition, we need the following two fundamental results of matroid
theory.
\begin{prop}[\edmonds][{\cite{ER66}}]\proplab{edmonds}
Let $f$ be a non-negative, increasing, integer-valued, submodular function on the power set of a finite set $E$.
Then the collection of subsets
\[
\mathcal{B}_f = \{ E'\subset E :\text{$f(E')=f(E)$ and for all $E''\subset E'$, $f(E'')\le \card{E''}$} \}
\]
gives the bases of a matroid defined to be $\mathcal{M}_f$.
\end{prop}

\begin{prop}[\edmondss][{\cite{ER66}}]\proplab{edmonds2}
Let $f$ be a non-negative, increasing, integer-valued, submodular function on the power set of a finite set $E$.
Then the matroid $\mathcal{M}_{2f}$ (in the sense of \propref{edmonds})
has as its bases the collection of subsets
\[
\mathcal{B}_{2f} = \{E'\subset E : \text{$E$ is the disjoint union of $2$ elements of $\mathcal{B}_{f}$}\}.
\]
\end{prop}

\subsection{Main combinatorial lemmas}
We have two key combinatorial results. The first one shows that the sparsity counts
give rise to a matroidal family of graphs.  We define the resulting matroids to be the
\emph{$(2,2,k)$-matroids}.

\begin{lemma}\lemlab{222matroid}
Let ${\bm{\gamma}}$ be a coloring of $K_{n}^{6,4}$ with $\mathbb{Z}^2$-rank $k$.  Then, the family of
$(2,2,k)$-colored-graphs, if non-empty, forms the bases of a matroid on $(K_{n}^{6,4},{\bm{\gamma}})$.
\end{lemma}

The second is the equivalent of the Nash-Williams-Tutte Theorem \cite{NW61,T61}
for $(2,2,\cdot)$-colored graphs. Instead of decompositions into spanning trees,
we seek decompositions into $(1,1,k)$-graphs, which we now define.  (See the figures in
\secref{11k} for some examples of $(1,1,k)$-graphs.)
\begin{figure}[htbp]
\centering
\subfigure[]{\includegraphics[width=.45\textwidth]{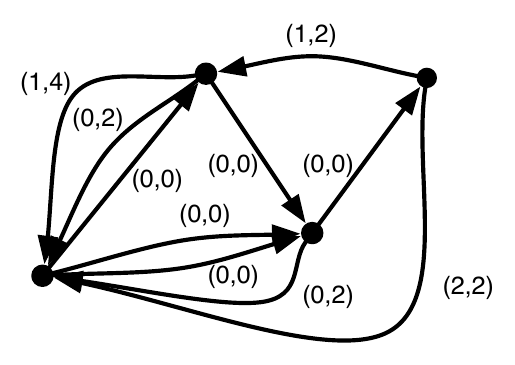}}
\subfigure[]{\includegraphics[width=.45\textwidth]{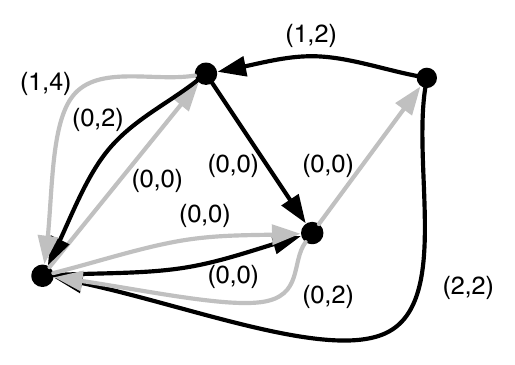}}
\caption{Example of a $(2,2,2)$-graph: (a) a $(2,2,2)$-graph; (b)
a decomposition into two $(1,1,2)$-graphs as provided by \lemref{222decomp}.
}
\label{fig:222-example}
\end{figure}

Let $(G,\bgamma)$ be a colored graph with $\Z^2$-rank $k$.  We define $(G,\bgamma)$ to be a\emph{ $(1,1,k)$-graph} if
$G$ is a spanning tree plus $k$ additional edges.  For the
purposes of this definition, we allow ``empty'' spanning trees when there is only one vertex.
\begin{lemma}\lemlab{222decomp}
A graph $(G,{\bm{\gamma}})$ is a $(2,2,k)$-graph if and only if it is the edge-disjoint union of
two spanning $(1,1,k)$-graphs.
\end{lemma}
\figref{222-example} shows an example of a $(2,2,2)$-graph and a certifying decomposition into
two $(1,1,2)$-graphs.

\subsection{Proof strategy for Lemmas \ref{lemma:222matroid} and \ref{lemma:222decomp}}
We will study the function
\[
g(G,\bgamma) = 2n + 2k - 2c
\]
where $n$, $c$ and $k$ have their usual meaning of the number of vertices, connected components,
and the $\Z^2$-rank. To show $g$ is submodular, we study the function $f = \frac{1}{2}g$,
which turns out to be the rank function of a matroid.  We then infer that
that $g$ satisfies the hypotheses of \propref{edmonds}, from
which \lemref{222matroid} and \lemref{222decomp} are  immediate.  We carry out this strategy in
Sections \ref{section:11k} and \ref{section:22k}.

Before moving on with the proof, we want to point out a subtlety in the transition from
$(2,2,2)$-graphs to the union of two $(1,1,2)$-graphs. The definition of $(2,2,2)$-graphs
makes clear that they have $\Z^2$-rank $2$.  However, \lemref{222decomp} says more: in fact,
any $(2,2,2)$-graph decomposes into two \emph{disjoint}, connected, spanning subgraphs that
each have $\Z^2$-rank $2$ in isolation.

\section{The $(1,1,k)$-matroids}\seclab{11k}
Let $(G,\bgamma)$ be a colored graph with $n$ vertices, $c$ connected components, and $\Z^2$-rank $k$.
Recall that we have defined the function
\[
f(G,{\bm{\gamma}}) = n + k - c
\]
which, after fixing a coloring of the ground set $(K_{n}^{6,4},{\bm{\gamma}})$, is a
function from subsets of the edges of $(K_{n}^{6,4})$ to the non-nagative integers.  We are going to prove:
\begin{lemma}\lemlab{fmat}
Let $\bgamma$ be a coloring of  $K_{n}^{6,4}$, and suppose that  $(K_{n}^{6,4},{\bm{\gamma}})$
has $\Z^2$-rank $k$.  Then:
\begin{itemize}
\item $f$ is non-negative, monotone, and submodular (i.e., \propref{edmonds} applies to it).
\item The matroid $M_f$ from \propref{edmonds} has as its bases $(1,1,k)$-graphs.
\end{itemize}
\end{lemma}

We start with some examples and simple structural results about $(1,1,k)$-graphs.
\subsection{Structure of $(1,1,k)$-graphs}
It is immediate from the definition of $(1,1,k)$-graphs and the $\Z^2$-rank that:
\begin{lemma}\lemlab{11kstruct}
Let $(G,\bgamma)$ be a $(1,1,k)$-graph.  Then,
\begin{itemize}
\item If $k=0$, then $G$ is a tree.
\item If $k=1$, then $G$ is a tree plus one additional edge,
and the unique cycle in $G$ has non-trivial $\Z^2$-image.
\item If $k=2$, then $G$ is a tree plus two additional edges, and
there are two cycles with linearly independent $\Z^2$-images.
\end{itemize}
\end{lemma}
\begin{figure}[htbp]
\centering
\includegraphics[width=.4\textwidth]{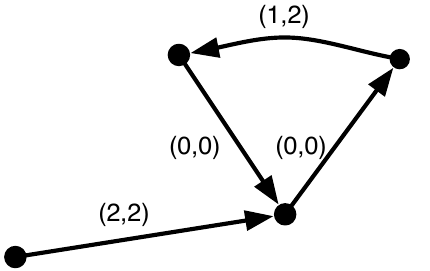}
\caption{A $(1,1,1)$-graph.}
\label{fig:111-example}
\end{figure}
\figref{111-example} shows an example of a $(1,1,1)$-graph.  \figref{112-example} shows an
example of two types of $(1,1,2)$-graphs.  It is not hard to see that every $(1,1,2)$-graph
contains a subdivision of either a vertex with two-self loops, three copies of a single edge (e.g., \figref{112-example} (a)),
or an edge with a self-loop on each endpoint (e.g., \figref{112-example} (b)).
\begin{figure}[htbp]
\centering
\subfigure[]{\includegraphics[width=.4\textwidth]{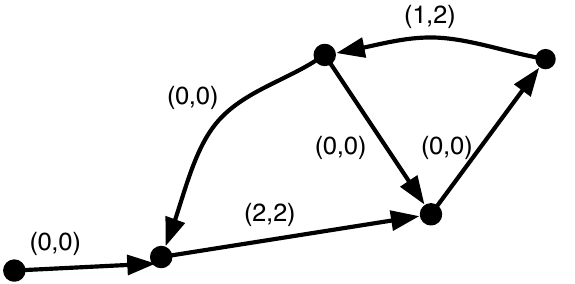}}
\subfigure[]{\includegraphics[width=.4\textwidth]{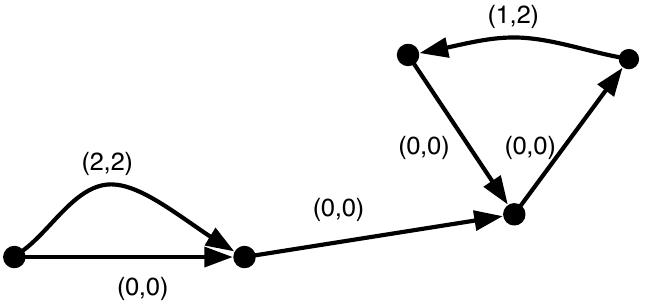}}
\caption{Examples of $(1,1,2)$-graphs.}
\label{fig:112-example}
\end{figure}

\subsection{An analogy to spanning trees}
If we consider the developments of $(1,1,k)$-graphs to periodic graphs, we can make an analogy
to connectivity in finite graphs.  We don't rely on \lemref{11k-lifts}, so we simply state it.
However, readers who are familiar with the role played by trees in proofs of combinatorial rigidity
characterizations may well find it instructive.
\begin{lemma}\lemlab{11k-lifts}
Let $(G,\bgamma)$ be a $(1,1,k)$-graph, and let $(\tilde{G},\varphi)$ be the development of
$(G,\bgamma)$.  Then removing any edge of $G$, disconnects every connected component of $\tilde{G}$.  Moreover, the
development is connected if and only if $k=2$ and $\rho(G,\bgamma)$ is all of $\Z^2$.
\end{lemma}
\lemref{11k-lifts} makes the connection between spanning trees (which are bases of the graphic matroid) and
$(1,1,2)$-graphs (which are bases of a matroid on colored graphs):
trees are minimally connected finite graphs and $(1,1,2)$-graphs have
minimally connected periodic developments if $\rho(G,\bgamma)$ generates $\Z^2$.

\subsection{Maximizers of $f$}
We now begin our study of the function $f$ in more detail, first considering the graphs maximizing the
function $f$.  These will turn out to be $(1,1,k)$-graphs for a $\Z^2$-rank $k$ coloring of $K_{n}^{6,4}$.

\begin{lemma}\lemlab{fmax-value}
Let $\bgamma$ be a coloring of  $K_{n}^{6,4}$, and suppose that  $(K_{n}^{6,4},{\bm{\gamma}})$
has $\Z^2$-rank $k$.  The maximum value of $f(G',\bgamma)$, over any colored sub-graph $(G',\bgamma)$
of $(K_{n}^{6,4},{\bm{\gamma}})$ is $n+k-1$.
\end{lemma}
\begin{proof}
Immediate from the definition.
\end{proof}

\begin{lemma}\lemlab{fmax-connected}
Let $\bgamma$ be a coloring of  $K_{n}^{6,4}$, and suppose that  $(K_{n}^{6,4},{\bm{\gamma}})$
has $\Z^2$-rank $k$.  Suppose that $(G',\bgamma)$ is a colored subgraph of $(K_{n}^{6,4},{\bm{\gamma}})$ and
$f(G',\bgamma)=n+k-1$.  Then $G'$ is connected.
\end{lemma}
\begin{proof}
$G'$ spans at most $n$ vertices and has $\Z^2$-rank at most $k$, and thus the contribution
of the positive terms to $f(G',\bgamma)$ is at most $n+k$.  For $f(G',\bgamma)$ to be $n+k-1$,
we must then have the number of connected components $c'$ equal to one.
\end{proof}

\begin{lemma}\lemlab{fmax-11k}
Let $\bgamma$ be a coloring of  $K_{n}^{6,4}$, and suppose that  $(K_{n}^{6,4},{\bm{\gamma}})$
has $\Z^2$-rank $k$.  Let  $(G',\bgamma)$ be any colored sub-graph of
$(K_{n}^{6,4},{\bm{\gamma}})$ with this coloring and $n-1+k$ edges.  Then $f(G',\bgamma)=n+k-1$ if
and only if $(G',\bgamma)$ is a $(1,1,k)$-graph.
\end{lemma}
\begin{proof}
First, we suppose that $f(G',\bgamma)=n+k-1$.  This means that \lemref{fmax-connected} applies, so $G'$ is connected, and
thus some $n-1$ of the edges of $G'$ form a spanning tree $T$ of $G'$.  If $k=0$, then we are done.  If $k=1$,
then there is one more edge $ij$, which creates a unique cycle $C$ composed of $ij$ and the path from $i$ to $j$
in $T$.  Since $f(G',\bgamma)=n$, $\rho(C)\neq (0,0)$, which implies that $(G',\bgamma)$ is a $(1,1,1)$-graph.  If $k=2$,
then $G'$ has two additional edges in addition to $T$, and $f(G',\bgamma)=n+1$ implies that there are two cycles with linearly
independent $\Z^2$-image in $(G',\bgamma)$, so it is a $(1,1,2)$-graph.

In the other direction, it is immediate from the definitions that if $(G',\bgamma)$ is a $(1,1,k)$-graph,
then $f(G',\bgamma)=n+k-1$.
\end{proof}

\subsection{Submodularity of $f$}
Next we show that $f$ meets the hypotheses of \propref{edmonds}.
We will verify the following form of the submodular inequality: let $(G',\bgamma)$ be a colored graph and let $(G'',\bgamma)$
be a subgraph of $(G',\bgamma)$.  To show that $f$ is submodular, it is enough to prove that for any colored edge $ij\notin E(G')$:
\[
f(G''+ij,{\bm{\gamma}}) - f(G'',{\bm{\gamma}}) \ge f(G'+ij,{\bm{\gamma}}) - f(G',{\bm{\gamma}})
\]
Before proving that $f$ is submodular, we need the following lemma.
\begin{lemma}\lemlab{f-increase}
Let $\bgamma$ be a coloring of  $K_{n}^{6,4}$.
Let  $(G',\bgamma)$ be any colored sub-graph of $(K_{n}^{6,4},{\bm{\gamma}})$,
and let $ij$ be any colored edge of the ground set not in $E(G')$.  Then $f(G'+ij,\bgamma)-f(G',\bgamma)$
is either zero or one.
\end{lemma}
\begin{figure}[htbp]
\centering
\subfigure[]{\includegraphics[width=.4\textwidth]{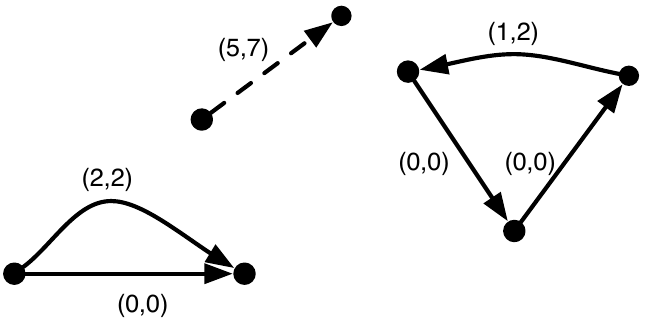}\label{fig:f-increase-1}}
\subfigure[]{\includegraphics[width=.4\textwidth]{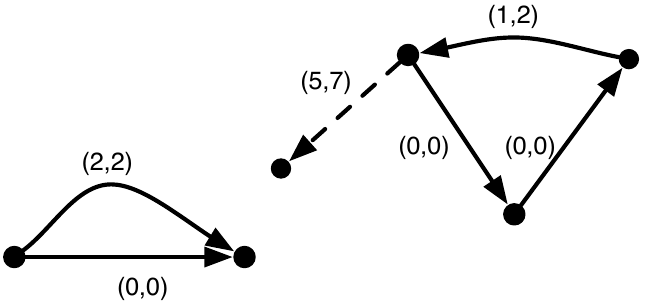}\label{fig:f-increase-2}}
\caption{Cases in the proof of \lemref{f-increase}: (a) Case 1; (b) Case 2.  The edge $ij$
to be added is indicated by a dashed line.}
\end{figure}
\begin{proof}
The proof is a case analysis based on how the new edge $ij$ interacts with the connected components of $G'$.  (We
remind the reader at this point that $ij$ may be a self-loop; i.e., $i$ may equal $j$.)

Let $n'$, $c'$ and $k'$ be the number of vertices, number of connected components and $\Z^2$-rank of $(G',\bgamma)$.
Similarly let $n''$, $c''$, $k''$  be the same quantities for $(G'+ij,\bgamma)$.

\noindent \textbf{Case 1:} \emph{$ij$ is disjoint from all connected components of $G$} (\figref{f-increase-1}).

If $ij$ is disjoint from $G'$ and $ij$ is an edge,
then $n''=n'+2$, $c''=c'+1$, and, since adding $ij$ cannot create any new cycles in $G'$,
$k'=k''$.  Thus we have $f(G'+ij,\bgamma)=n''+k''-c''=n'+2+k'-(c'+1)=f(G',\bgamma)+1$.

If $ij$ is a self-loop, then $n''=n'+1$ and $c''=c'+1$, so the only possible change can
come from $k''-k'$.  This is either zero or one, depending on whether the color
$\gamma_{ij}$ is in the span of $\rho(G',\bgamma)$.

\noindent \textbf{Case 2:} \emph{$i$ is in some connected component of $G'$ and $j$ is not} (\figref{f-increase-2}).

In this case, $j$  becomes a leaf ($ij$ cannot be a self loop in this case).
We have $n''=n'+1$, $c'=c''$, and, since no cycle is created, $k''=k'$.  It follows that
$f(G'+ij,\bgamma)=f(G',\bgamma)+1$.
\begin{figure}[htbp]
\centering
\subfigure[]{\includegraphics[width=.4\textwidth]{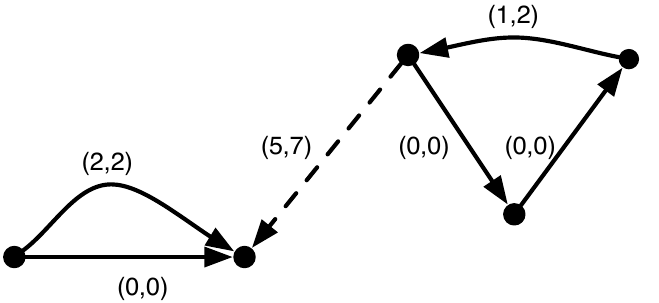}\label{fig:f-increase-3}}
\subfigure[]{\includegraphics[width=.4\textwidth]{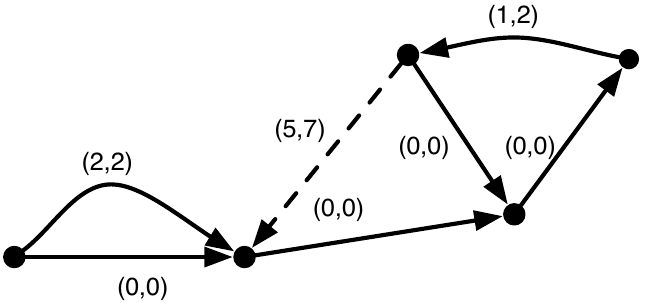}\label{fig:f-increase-4}}
\caption{Cases in the proof of \lemref{f-increase}: (a) Case 3; (b) Case 4.  The edge $ij$
to be added is indicated by a dashed line.}
\end{figure}

\noindent \textbf{Case 3:} \emph{$i$ and $j$ are in different connected components of $G'$} (\figref{f-increase-3}).

In other words, $ij$ is a bridge in $G+ij$ (and so $ij$ cannot be a self-loop).
No new cycles are created, so $k'=k''$.  Since $c''=c'-1$ and $n'=n''$
$f(G'+ij,\bgamma)=f(G',\bgamma)+1$.

\noindent \textbf{Case 4:} \emph{$ij$ is in the span of some connected component of $G$} (\figref{f-increase-4}).

The number of vertices and connected components is fixed, and the proof follows from \lemref{z2-rank-01}.  (The
treatment of self-loops is uniform in this case.)
\end{proof}

\begin{lemma}\lemlab{submodular}
Let $\bgamma$ be a coloring of  $K_{n}^{6,4}$.  The function $f$ from subgraphs $(G',\bgamma)$ of
$(K_{n}^{6,4},\bgamma)$ is submodular.
\end{lemma}
\begin{proof}
We check the submodular inequality.  To do this, we let $(G',\bgamma)$ be a colored subgraph of the ground
set, $(G'',\bgamma)$ be a subgraph of $(G',\bgamma)$, and $ij$ a colored edge not in $(G',\bgamma)$.  We
need to show that
\[
f(G''+ij,{\bm{\gamma}}) - f(G'',{\bm{\gamma}}) \ge f(G'+ij,{\bm{\gamma}}) - f(G',{\bm{\gamma}})
\]
By \lemref{f-increase}, both sides of the inequality to prove are either zero or one.  It follows that
when the r.h.s. is zero, we are done.  What's left is to assume that the r.h.s. is one, and show that
this implies that the l.h.s. is as well.

Since $f(G'+ij,{\bm{\gamma}}) - f(G',{\bm{\gamma}})=1$, we know from the proof of \lemref{f-increase}
that this was due either to:
\begin{itemize}
\item One of Cases \textbf{1}, \textbf{2}, or \textbf{3}.
\item Case \textbf{4}, where the $\Z^2$-rank increased.  (This is the only possibility if $ij$ is a
self-loop.)
\end{itemize}
Because $G''$ is a subgraph of $G'$, if the increase $f(G'+ij,{\bm{\gamma}}) - f(G',{\bm{\gamma}})=1$ is due to
Cases \textbf{1}, \textbf{2}, or \textbf{3}, then $ij$ is not in the span of any connected component
of $G''$ either.  This means that adding $ij$ to $G''$ will force one
of these cases as well, and the desired inequality follows.

To complete the proof, we suppose that the increase
$f(G'+ij,{\bm{\gamma}}) - f(G',{\bm{\gamma}})=1$ came from Case \textbf{4}.  Again, if $ij$ was not
in the span of any connected component of $G''$, this forces one of the first three cases analyzed
in the proof of \lemref{f-increase}, and we are done.  Otherwise, all the hypotheses of
\lemref{z2-rank} are met, and the lemma follows.
\end{proof}
We remark that we have shown more about $f$ than was strictly necessary to apply \propref{edmonds}.  As
readers familiar with matroid theory will have noticed, \lemref{f-increase} and \lemref{submodular}
together imply that $f$ is actually the rank function of a matroid.

\subsection{Proof of \lemref{fmat}}
The non-negativity and monotonicity of the function $f$ follow from \lemref{f-increase} and submodularity
was checked in \lemref{submodular}.  Thus \propref{edmonds} applies to $f$, and we have the first statement.

For the second statement, \lemref{fmax-11k} implies that the bases of the resulting matroid $\mathcal{M}_f$
are $(1,1,k)$-graphs.
\hfill $\qed$

\section{The $(2,2,k)$-matroids}\seclab{22k}
We have the tools is place to prove our main combinatorial results on $(2,2,k)$-graphs.

\subsection{Proof of Lemmas \ref{lemma:222matroid} and \ref{lemma:222decomp}}
Fix a coloring ${\bm{\gamma}}$.  On the one hand, since \lemref{submodular} implies that $f$  meets the
conditions of \propref{edmonds}, $2f$ does too, so $\mathcal{M}_{2f}$ has as its bases $(2,2,k)$-colored graphs if
the $\Z^2$-rank of $(K_n^{6,4},{\bm{\gamma}})$ is $k$, since the right hand side of the defining
sparsity condition is just $2f(G',{\bm{\gamma}})$. This proves \lemref{222matroid}.

On the other hand, \propref{edmonds2} says that the bases of $\mathcal{M}_{2f}$ are exactly the graphs
that decompose into two disjoint bases of $\mathcal{M}_f$.  By \lemref{fmat}, these are the
decompositions required by \lemref{222decomp}.
\hfill $\qed$

\section{The colored-Laman matroid}\seclab{cloning}
Although it is a corollary of
\theoref{periodiclaman}, we can give a proof of the following lemma directly.
\begin{lemma}\lemlab{232matroid}
Let $(K_n^{6,4},{\bm{\gamma}})$ have $\mathbb{Z}^2$-rank 2.
Then the family of colored-Laman graphs forms the bases of a matroid.
\end{lemma}
\begin{proof}
This is an application of \propref{edmonds}	to the function
\[
h(G',{\bm{\gamma}}) = 2n'- 3 + 2k' - 2(c'-1)
\]
where $n'$, $c'$, and $k'$ have their usual meanings.  Since the function $h$ is
$2f(G',\bgamma)-1$, and $f$ is submodular and monotone,  $h$ is as well. Non-negativity
of $h$ is not hard to check.
\end{proof}
As discussed in the introduction, \lemref{232matroid} amounts to saying that the colored-Laman
matroid is the (combinatorial) Dilworth Truncation of the $(2,2,2)$-matroid.

\subsection{Circuits in the colored-Laman matroid}
If a colored graph $(G,\bgamma)$ is not colored-Laman-sparse, then it must have some subgraph $(G',\bgamma)$
with $m'$ edges and $m' > h(G',\bgamma)\ge 2f(G,\bgamma)$, where $g$ is
defined above and $f$ is the function defined in \secref{tdngraphs}.

We define a colored graph $(G,\bgamma)$ to be a \emph{colored-Laman circuit} if:
\begin{itemize}
\item $G$ has  $m=2f(G,\bgamma)$ edges.
\item Removing \emph{any} edge $ij$ from $G$ results in a colored-Laman-sparse
graph $(G-ij,\bgamma)$.
\end{itemize}
We note that $(G-ij,\bgamma)$ may not be a colored-Laman graph.  For example \figref{dependent-framework}
shows a colored-Laman circuit, but it has no spanning subgraph that is colored-Laman.

The colored-Laman circuits are the minimal obstructions to colored-Laman-sparsity.
\begin{lemma}\lemlab{circuits}
Let $(G,{\bm{\gamma}})$ be a colored graph, and suppose that $(G,{\bm{\gamma}})$ is not
colored-Laman-sparse.  Then $(G,{\bm{\gamma}})$ has
as an edge-induced subgraph $(G',{\bm{\gamma}})$ a colored-Laman circuit.
\end{lemma}
\begin{proof}
We begin by extracting a maximal subgraph $(B,{\bm{\gamma}})$ of $(G,{\bm{\gamma}})$ that is
colored-Laman sparse.  The matroidal property (\lemref{232matroid}) implies that
these all have the same size, and since  $(G,{\bm{\gamma}})$ is not colored-Laman-sparse,
$(B,{\bm{\gamma}})$ is not all of $(G,{\bm{\gamma}})$.  Let $ij$ be an edge not in $B$, and
consider the subgraph $(B+ij,{\bm{\gamma}})$.

Since each subgraph of $(B,{\bm{\gamma}})$ gained at most one more edge, $(B+ij,{\bm{\gamma}})$ is
colored-$(2,2,2)$ sparse and since $(B+ij,{\bm{\gamma}})$
is not colored-Laman-sparse, some subgraph $(C,{\bm{\gamma}})$ of it on $n'$ vertices,
$m'$ edges, $c'$ connected components, and $\Z^2$-rank $k'$ must have
\[
m' = 2n'  + 2k'-2c' = 2f(G')
\]
A final appeal to the matroidal property of colored-Laman graphs shows that there is a unique minimal
such $(C,{\bm{\gamma}})$, which will be a colored-Laman circuit as desired.
\end{proof}

\subsection{Characterization of colored-Laman graphs by edge doubling}
The following characterization of colored-Laman graphs is very similar in spirit to the
Lovász-Yemini \cite{LY82} and Recski \cite{R84} characterizations of Laman graphs, and the proof
is similar as well.
\begin{lemma}\lemlab{periodiclamancontract}
Let $(G,{\bm{\gamma}})$ be a colored graph with $n$ vertices and $2n+1$ edges.  Then $(G,{\bm{\gamma}})$
is colored-Laman if and only if doubling any edge $ij$ results in a $(2,2,2)$-colored-graph
$(G+(ij)_c,{\bm{\gamma}})$.
\end{lemma}
\begin{proof}
First suppose that $(G,{\bm{\gamma}})$ is colored-Laman.  Any subgraph of $G+(ij)_c$ that is a subgraph of
$G$ already satisfies the sparsity counts.  Suppose then that a subgraph $G'$ of $G+(ij)_c$ contains $(ij)_c$.
If $G'$ also contains the edge $ij$, then $G'$ is $G' - (ij)_c$ plus one edge (which adds no new vertices or
components to $G' - (ij)_c$).  Since $G' - (ij)_c$ is a subgraph of $G$, if $G' - (ij)_c$ has $m''$ edges,
$n''$ vertices, $c''$ components and $\Z^2$-rank $k''$, then
$$m'' \leq 2n'' - 3 + 2k'' - 2(c''-1).$$  Observe that $G'$ has $m' = m''+1$ edges, $n' = n''$ vertices, $c' = c''$
components and $\Z^2$-rank $k' \geq k''$ and hence $m' \leq 2n' - 2 + 2k' - 2(c'-1)$.  If $G'$ does not contain the
edge $ij$, then $G' - (ij)_c + ij$, a subgraph of $G$, has the same rank and number of vertices, edges,
and components as $G'$.

On the other hand, if $(G, {\bm{\gamma}})$ is not colored-Laman, then there is some subgraph
$(G',{\bm{\gamma}})$ with $n'$ vertices, $\Z^2$-rank $k'$, and at least $2n' - 2 + 2k' - 2(c'-1)$
edges.  Then doubling any edge in $G'$ results in a graph that is not $(2,2,2)$-colored.
\end{proof}

\section{Natural representations}\seclab{natural}
In Sections \ref{section:22k} and \ref{section:cloning},
we proved that the $(2,2,k)$-colored graphs and colored-Laman graphs on $n$ vertices
each give the bases of a matroid.
In matroidal terms, the rigidity \theoref{periodiclaman} states that
the \emph{rigidity matrix} (defined in \secref{infinitesimal})
for generic periodic bar-joint frameworks \emph{represents} the
colored-Laman matroid: linear independence among the rows of the matrix corresponds bijectively
to independence in the associated combinatorial matroid.

The next step in the program set out in the introduction is to give
linear representations of the $(2,2,k)$-colored matroids which are \emph{natural} in the sense
that the matrices obtained have the same dimensions as the corresponding rigidity matrices and
non-zero entries at the same positions.

We now give the detailed definitions and state the main representation
results on $(2,2,k)$-matroids.

\subsection{The generic rank of a matrix}
A \emph{generic matrix} has as its non-zero entries \emph{generic variables}, or formal
polynomials over $\mathbb{R}$ in generic variables.  Its \emph{generic rank}
is given by the largest number $r$ for which $\vec M$ has an $r\times r$ matrix minor with a
determinant that is formally non-zero.

Let $\vec M$ be a generic matrix in $m$ generic variables $x_1,\ldots, x_m$, and
let $\vec v=(v_i)\in \mathbb{R}^m$.  We define a \emph{realization $\vec M(\vec v)$ of $\vec M$}
to be the matrix obtained by replacing the variable $x_i$ with the corresponding number $v_i$.  A
vector $\vec v$ is defined to be a \emph{generic point} if the rank of $\vec M(\vec v)$ is equal to the
generic rank of $\vec M$; otherwise $\vec v$ is defined to be a \emph{non-generic} point.

We will make extensive use of the following well-known facts from algebraic geometry (see, e.g., \cite{CLO97}):
\begin{itemize}
\item The rank of a generic matrix $\vec M$ in $m$ variables
is equal to the maximum over $\vec v\in \mathbb{R}^m$ of the rank of all
realizations $\vec M(\vec v)$.
\item The set of non-generic points of a generic matrix $\vec M$ is an
algebraic subset of $\mathbb{R}^m$.
\item The rank of a generic matrix $\vec M$ in $m$ variables is at least
as large as the rank of any specific realization $\vec M(\vec v)$; i.e., generic rank
can be established by a single example.
\end{itemize}

\subsection{Generic representations of matroids}
Let $\mathcal{M}$ be a matroid on ground set $E$.  We define a generic matrix $\vec M$ to be
a \emph{generic representation of $\mathcal{M}$} if:
\begin{itemize}
\item There is a bijection between the rows of $\vec M$ and the ground set $E$.
\item A subset of rows of $\vec M$  attains the rank of the matrix $\vec M$
if and only if the corresponding subset of $E$ is a basis of $\mathcal{M}$.
\end{itemize}

\subsection{The natural representation of the $(2,2,k)$-matroids}
Let $(G, \bm{\gamma})$ be a colored graph with $n$ vertices and $m$ edges.  We define
the matrix  $\vec M_{2,2,2}(G, {\bm{\gamma}})$ to be the $m$ by $2n+4$ matrix with the
filling pattern indicated below:
$$  \bordermatrix{                   &              &     i    &             &     j &            &           L_1                               & L_2  \cr
& \dots & \dots               & \dots & \dots               & \dots & \dots                                      & \dots \cr
ij   &  0\dots0 & -a_{ij} \;\; -b_{ij} & 0\dots0 & a_{ij} \;\; b_{ij} & 0\dots 0& \gamma_{ij}^1 a_{ij} \;\; \gamma_{ij}^1 b_{ij} & \gamma_{ij}^2 a_{ij} \;\; \gamma_{ij}^2 b_{ij}\cr
& \dots & \dots                & \dots & \dots               & \dots & \dots                                     & \dots  }.
$$
The rows of $\vec M_{2,2,2}(G,{\bm{\gamma}})$	are indexed by the edges $ij\in E(G)$.  The columns are indexed as follows:
the first $2n$ columns are indexed by the vertices $V(G)$, with two columns for every vertex; the last $4$
are associated with the coordinate projections of ${\bm{\gamma}}$ onto $\mathbb{Z}$, with each getting two columns $L_1$
and two columns $L_2$.  The entries $a_{ij}$ and $b_{ij}$ are generic variables, with different copies of an edge
getting distinct variables.  Notice that the sign pattern of the matrix encodes the underlying orientation
of $G$.

We now state our representation result for the $(2,2,k)$-matroid.
\begin{lemma} \lemlab{M2rank}
Let $(G, \bm{\gamma})$ be a colored graph with $2n - 2 + 2k$ edges and $\Z^2$-rank $k$.  Then $(G, \bm{\gamma})$ is $(2, 2, k)$-colored
if and only if $\vec M_{2, 2, 2}(G, \bm{\gamma})$ has generic rank $2n-2+ 2k$.
\end{lemma}
From this we have the immediate corollary:
\begin{cor}\corlab{M2rep}
If $(K_n^{6, 4}, \bm{\gamma})$ has $\Z^2$-rank $2$, then the matrix $\vec M_{2,2,2}( K_n^{6,4},{\bm{\gamma}})$ represents the $(2, 2, 2)$-colored matroid.
\end{cor}
To prove \lemref{M2rank}, we first establish analogous results for $(1,1,k)$-graphs in \secref{11k-rep}.  This
is done using determinant formulas similar to standard ones for the graphic matroid.  The deduction, in \secref{22k-rep},
of \lemref{M2rank} from the results of \secref{11k-rep} and \lemref{222decomp} is then a standard argument that is
nearly the same as that for the Matroid Union Theorem for representable matroids \cite[Proposition 7.6.14]{B86} or
the specialization to the graphic matroid \cite[Theorem 1]{W88}.

\section{Natural representations of $(1,1,k)$-graphs}\seclab{11k-rep}
Let $(G,{\bm{\gamma}})$ be a colored graph, and define the matrix $M_{1,1,2}(G,{\bm{\gamma}})$ to
have the filling pattern indicated below:
$$
\bordermatrix{                   &              &     i    &             &     j &            &           L                                \cr
& \dots & \dots               & \dots & \dots               & \dots & \dots                                       \cr
ij   &  0\dots0 & -a_{ij}  & 0\dots0 & a_{ij} & 0\dots 0& \gamma_{ij}^1 a_{ij} \;\; \gamma_{ij}^2 a_{ij}\cr
& \dots & \dots                & \dots & \dots               & \dots & \dots                                     & \dots  }.
$$
The row and column indexing is similar to that for
$\vec M_{2,2,2}(G, {\bm{\gamma}})$: there are $m$ rows, one for each edge, and $n+2$ columns,
one for each vertex and two associated with the coordinate projections of the coloring.

\begin{lemma}\lemlab{112total}
Let $(G, \bm{\gamma})$ be a colored graph with $n - 1 + k$ edges.  Then $(G, \bm{\gamma})$ is $(1, 1, k)$-colored
if and only if $M_{1, 1, 2}(G, \bm{\gamma})$ has generic rank $n - 1 + k$.
\end{lemma}
\begin{proof}
Lemmas \ref{lemma:110rep}--\ref{lemma:112rep} (below) prove the lemma for each rank.
\end{proof}
The rest of the section contains the proofs.

\subsection{Natural representation of the graphic matroid}
The sub-matrix induced by the $n$ columns is standard in matroid theory: it is the usual
generic representation of the graphic matroid.  We will denote this sub-matrix $\vec M_{1,1}(G)$ (there is
no dependence on the coloring, so the notation suppresses it).  The following lemma is standard.
\begin{lemma}[][{\cite[Section 4]{L07}}]\lemlab{M11}
Let $G$ be a graph with $n$ vertices and $n-1$ edges.  Let $\vec M^\bullet_{1,1}(G)$ be a (necessarily square)
matrix minor of $\vec M_{1,1}(G)$ obtained by dropping any one column.  Then:
\begin{itemize}
\item $\det(\vec M^\bullet_{1,1}(G)) = \pm\prod_{ij\in E(G)} a_{ij}$ if $G$ is a tree.
\item $\det(\vec M^\bullet_{1,1}(G)) = 0$ otherwise.
\end{itemize}
\end{lemma}

\subsection{Cycle elimination lemma}
The defining property of $(1,1,k)$-graphs is the presence of $k$ cycles with non-trivial, independent $\rho$-images.
The following lemma shows how to make the connection explicit.
\begin{lemma}\lemlab{cycle}
Let $(G,{\bm{\gamma}})$ be a colored graph on $n$ vertices with a cycle $C$ and an edge $i_0j_0$ on $C$.
Then $\vec M_{1,1,2}(G,{\bm{\gamma}})$ can be put into a form by elementary operations such that:
\begin{itemize}
\item The row $\vec r_{i_0j_0}$ corresponding to $i_0j_0$ has all zeros in the first $n$
columns and $\rho(C)$ in the last two columns.
\item All other rows have entries $\pm 1$ or $0$
in the first $n$ columns and $\gamma_{ij}$ in the last two columns.
\end{itemize}
\end{lemma}
\begin{proof}
First scale each row by $1/a_{ij}$.  	It is easy to check (and is a standard result, e.g., \cite[Problem 16-3]{CLRS})
that after scaling every row by $1/a_{ij}$:
\[
\vec r_{i_0j_0} - \sum_{\substack{ij\neq i_0j_0\in C\\ \text{traversed forwards}}} \vec r_{ij} +
\sum_{\substack{ij\neq i_0j_0\in C\\ \text{traversed backwards}}} \vec r_{ij}
\]
equals a row vector with zeros in the first $n$ columns and $\rho(C)$ in the last two columns.
\end{proof}
The next several lemmas establish natural representation of $(1,1,k)$-colored graphs.

\subsection{Natural representation for $(1,1,0)$-graphs}
\begin{lemma}\lemlab{110rep}
Let $(G,{\bm{\gamma}})$ have $\mathbb{Z}^2$ rank 0, $n$ vertices and $m=n-1$ edges.  Then:
\begin{itemize}
\item $\vec M_{1,1,2}(G,{\bm{\gamma}})$ has generic rank $n-1$ if and only if $(G,{\bm{\gamma}})$ is
$(1,1,0)$-colored (i.e., $G$ is a tree).
\item The minor $\vec M^\bullet_{1,1,2}(G,{\bm{\gamma}})$ obtained by dropping both columns in
the $L$ block and any of the first $n$ columns indexed by vertices has determinant:
\[
\det\left(\vec M^\bullet_{1,1,2}(G,{\bm{\gamma}}) \right)=\pm\prod_{ij\in E}a_{ij}
\]
if $G$ is a tree and $0$ otherwise.
\end{itemize}
\end{lemma}
\begin{proof}
First we suppose that $G$, the underlying graph is a tree.  In this case, dropping in the $L$ block of
$\vec M_{1,1,2}(G,{\bm{\gamma}})$ and one other column leaves the matrix $\vec M^\bullet_{1,1}(G)$,
defined above.  \lemref{M11} then implies both the desired rank and determinant formula.

Now suppose that $G$ is not tree.  We check that any $(n-1)\times (n-1)$ minor of
$\vec M_{1,1,2}(G,{\bm{\gamma}})$ has a determinant that is formally zero.
Since $G$ is not a tree, it must contain a cycle $C$.  After applying
\lemref{cycle} to $(G,{\bm{\gamma}})$ we obtain a matrix with a row of all zeros,
since $\rho(C)=(0,0)$ by hypothesis.  Thus $\vec M_{1,1,2}(G,{\bm{\gamma}})$ is rank
deficient as desired.
\end{proof}

\subsection{Natural representation for $(1,1,1)$-graphs}
\begin{lemma}\lemlab{111rep}
Let $(G,{\bm{\gamma}})$ have $\mathbb{Z}^2$ rank 1, $n$ vertices and $m=n$ edges.  Then:
\begin{itemize}
\item $\vec M_{1,1,2}(G,{\bm{\gamma}})$ has generic rank $n$ if and only if $(G,{\bm{\gamma}})$
is a $(1,1,1)$-graph.
\item The minor $\vec M^\bullet_{1,1,2}(G,{\bm{\gamma}})$ obtained by dropping some column in the $L$ block and any
of the first $n$ columns indexed by vertices has determinant:
\[
\det\left(\vec M^\bullet_{1,1,2}(G,{\bm{\gamma}}) \right)=\pm t_q\prod_{ij\in E}a_{ij}
\]
if $G$ is $(1,1,1)$-colored and $C$ is its unique cycle with $\rho(C)=(t_1,t_2)$ ($q\in \{1,2\}$), and $0$ otherwise.
\end{itemize}
\end{lemma}
\begin{proof}
Suppose that $(G,{\bm{\gamma}})$ is a (1,1,1)-graph with cycle $C$ spanning an edge
$i_0j_0$.  Applying \lemref{cycle} gives $\vec M_{1,1,2}(G,{\bm{\gamma}})$
a form where: $\vec r_{i_0j_0}$ has zeros in the first $n$ columns and $\rho(C)$ in the $L$ columns;
and all other entries in the first $n$ columns are zero or $\pm 1$.
Since $\rho(C)=(t_1,t_2)\neq (0,0)$ assume, w.l.o.g., that $\rho_1\neq 0$.
Consider the minor formed by dropping any of the first $n$ columns and the second column from
the block $L$.

The determinant of this minor (by expanding along the remaining column from $L$)
is $\rho_1\cdot \det(\vec A)$, where $\vec A$, the complementary cofactor, is $\vec M^\bullet_{1,1}(G-i_0j_0)$.
Because $(G,\bgamma)$ is a $(1,1,1)$-graph, \lemref{11kstruct} implies that $G-i_0j_0$ is a tree, and so
$\det(\vec A)=\pm 1$ by \lemref{M11}.

The desired determinant formula then follows from noting that the effect on the determinant by the scaling
in \lemref{cycle} is to multiply it by $\prod_{ij\in E(G)}a_{ij}$.

If $(G,{\bm{\gamma}})$ is not (1,1,1)-colored, it has more than one cycle.  Let $C_1$ and $C_2$ be two such cycles
with $\rho(C_r)=(t_1^r,t_2^r)$ ($r=1,2$).  Since there are $ij \in C_1$
with $ij \notin C_2$ and $i'j' \in C_2$ with $i'j' \notin C_1$, we can apply \lemref{cycle} two times to put
$\vec M_{1,1,2}(G,{\bm{\gamma}})$ into a form where the $ij$ row has
all zeros in the first $n$ columns and $\rho(C_1)$ in the last two and the $i'j'$ row has zeros in the
first $n$ columns and $\rho(C_2)$ in the last two.

There are two types of $n\times n$ minors to consider:
\begin{itemize}
\item Minors with both columns from $L$ have determinant $\pm(t_1^1 t_2^2-t_1^2 t_2^1)$, but this is zero, since the
$\mathbb{Z}^2$-rank 1 hypothesis implies that $\rho(C_1)$ and $\rho(C_2)$ are linearly dependent.
\item Minors with at most one column from $L$.  The determinant is zero, since either the minor has a row of all zeros
or the cofactor in the expansion on the remaining column of $L$ always does.
\end{itemize}
\end{proof}

\subsection{Natural representation for $(1,1,2)$-graphs}
\begin{lemma}\lemlab{112rep}
Let $(G,{\bm{\gamma}})$ have $\mathbb{Z}^2$ rank 2, $n$ vertices and $m=n+1$ edges.  Then:
\begin{itemize}
\item $\vec M_{1,1,2}(G,{\bm{\gamma}})$ has generic rank $n$ if and only if $(G,{\bm{\gamma}})$ is a $(1,1,2)$-graph.
\item The minor $\vec M^\bullet_{1,1,2}(G,{\bm{\gamma}})$ obtained by dropping any
of the first $n$ columns indexed by vertices has determinant:
\[
\det\left(\vec M^\bullet_{1,1,2}(G,{\bm{\gamma}}) \right)=\pm(t_1^1 t_2^2-t_1^2 t_2^1)\prod_{ij\in E}a_{ij}
\]
if $G$ is $(1,1,1)$-colored, with cycles $C_1$ and $C_2$, $\rho(C_q)=(t_1^q,t_2^q)$  ($q\in \{1,2\}$) and $0$ otherwise.
\end{itemize}
\end{lemma}
\begin{proof}
If $(G,{\bm{\gamma}})$ is $(1,1,2)$-colored than it has two cycles
$C_1$ and $C_2$, with linearly independent $\rho$ images.  The
structural \lemref{11kstruct} for $(1,1,2)$-graphs implies that we can
find edge $i_{1}j_1$ and $i_{2}j_2$ on cycle $C_1$ but not $C_2$ and $C_2$ but not $C_1$,
respectively.

Adopting the arguments and notation from the proof of \lemref{111rep}, we get the desired determinant formula,
since the $\Z^2$-rank $2$ hypothesis implies that  $(t_1^1 t_2^2- t_1^2 t_2^1)\neq 0$.

On the other hand, if $(G,{\bm{\gamma}})$ fails to be $(1,1,2)$-colored, then we can iteratively apply \lemref{cycle} three times to
put it $\vec M_{1,1,2}(G,{\bm{\gamma}})$ in a form where there are three rows with all zeros in the first $n$ columns,
which shows it to be rank deficient.
\end{proof}
The proof of \lemref{112total} is now complete.

\subsection{Maximum rank lemma}
We conclude with a small result about the maximum rank of $\vec M_{1,1,2}(G,\bgamma)$.
\begin{lemma}\lemlab{112max}
Let $(G, \bm{\gamma})$ have $\mathbb{Z}^2$-rank $k$ and $m > n- 1 + k$ edges.
Then $\vec M_{1,1,2}(G,\bm{\gamma})$ has a row dependency.
\end{lemma}
\begin{proof}
If $m> n+2$ this follows from the shape.  The other case is where $m\le n+2$.
The edge counts imply that \lemref{cycle} can be applied $k+1$ times to leave
$k+1$ rows with all zeros in the first $n$ columns.  Using the same arguments
as above, the determinant of any $m\times m$ submatrix is zero.
\end{proof}

\section{Natural representations of $(2,2,k)$-graphs}\seclab{22k-rep}
The main step in the proof of  \lemref{M2rank} is to determine the rank of $\vec M_{2, 2, 2}(G,\bgamma)$
when $(G,\bgamma)$ decomposes into two $(1,1,2)$-graphs.
\begin{lemma}\lemlab{rankdecomp}
Let $(G, \bm{\gamma})$ be a colored graph with $2n - 2 + 2k$ edges and $\Z^2$-rank $k$.  Then $(G, \bm{\gamma})$ is
the edge-disjoint union of two $(1,1,k)$-graphs if and only if $\vec M_{2, 2, 2}(G,\bgamma)$ has generic rank $2n-2+ 2k$.
\end{lemma}
The proof is quite similar to that from \cite[Proposition 7.16.4]{B86} for Matroid Union for
linearly-representable matroids.
\begin{proof}
Let $(G,{\bm{\gamma}})$ be a $\mathbb{Z}^2$-rank $k$ colored graph with $m=2n-2+2k$ edges.
Let $\vec M^\bullet$ be any $m\times m$ submatrix of $\vec M_{2,2,2}(G,{\bm{\gamma}})$, and
let $A$ be the set of columns of $\vec M^\bullet$
with $a_{ij}$ and $B$ be the set of columns with $b_{ij}$.  We compute the determinant using the Laplace expansion:
\[
\det \left(\vec M^\bullet\right) = \sum_{\substack{X\subset [m] \\ \card{X}=\card{A}}} \pm\det\left(\vec M^\bullet[X,A] \right)\cdot\det\left(\vec M^\bullet[[m]-X,B] \right)
\]
The key observation is that each of the sub-determinants  in the sum has the form of a minor of $\vec M_{1,1,2}(G',{\bm{\gamma}})$ for
an edge-induced subgraph of $(G,{\bm{\gamma}})$, and the sub-determinants correspond to disjoint subgraphs.
First note that, \lemref{112total} and \lemref{112max} imply that unless $\vec M^\bullet$ was obtained by dropping: one column from
the first $n$ in each of $A$ and $B$ and $2-k$ columns from $L_1$ and $L_2$ in each $A$ and $B$; at least one of
the sub-determinants is zero in every term.

Now consider an $\vec M^\bullet$ of the form described.
By \lemref{112total}, unless both $X$ and $[m]-X$ correspond to $(1,1,k)$-colored subgraphs
of $(G,{\bm{\gamma}})$, every term in the determinant expansion has a zero factor,
so the whole determinant is zero.  On the other hand, if there is such a decomposition, then the whole
determinant cannot cancel, since combinatorially different decompositions give rise to
combinatorially different monomials in the $a_{ij}$ and $b_{ij}$.
\end{proof}

\subsection{Proof of \lemref{M2rank}}
The lemma is immediate from \lemref{rankdecomp} and the $(2,2,k)$-graph
decomposition \lemref{222decomp}.
\hfill $\qed$

\section{Periodic rigidity on the line}\seclab{line}
As a warm up result, we will give a combinatorial characterization of periodic rigidity on the Euclidean
line $\R$.  The definitions of frameworks and their associated colored quotient graphs are
specializations of those for the planar case:
\begin{itemize}
\item An abstract \emph{1d-periodic framework} $(\tilde{G},\varphi,\tilde{\bm{\ell}})$ is given by an infinite graph
with a free $\Z$-action that has finite quotient,
and an assignment $\tilde{\bm{\ell}}=(\tilde{\ell_{ij}})_{ij\in E(\tilde{G})}$ of \emph{edge-lengths}
that respects the action $\varphi$.
\item A \emph{realization} $\tilde{G}(\vec p,L)$ of the abstract framework is a
mapping of $V(\tilde{G})$ onto a periodic point set $\vec p=(x_i)_{i\in V(\tilde{G})}$
such that the edge lengths are respected.
\end{itemize}
The relationship between a 1d-periodic framework and the associated quotient graph (which will
have colors in $\Z$) is also similar to the planar case.  \figref{1d-framework} shows an example.

Rigidity and flexibility are also defined (on realizations) in a similar way: a framework is rigid if the
only allowed continuous motions are translations of the line and otherwise flexible.

In this section, we will show:
\begin{theorem}\theolab{1d-periodiclaman}
Let $(\tilde G, \varphi,\tilde{\bm{\ell}})$ be a generic  1d-periodic framework.  Then a generic
realization $\tilde{G}(\vec p,L)$ of $(\tilde G, \varphi,\tilde{\bm{\ell}})$ is
minimally rigid if and only if its quotient graph $(G,{\bm{\gamma}})$ is a $(1,1,1)$-graph.
\end{theorem}
The analogous result for finite frameworks on the line is that a framework is minimally
rigid if and only if the graph formed by the bars is a tree
(see, e.g., \cite[Section 2.5]{G01}).
\begin{figure}[htbp]
\centering
\subfigure[]{\includegraphics[width=.25\textwidth]{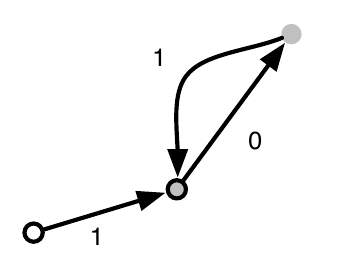}}
\subfigure[]{\includegraphics[width=.8\textwidth]{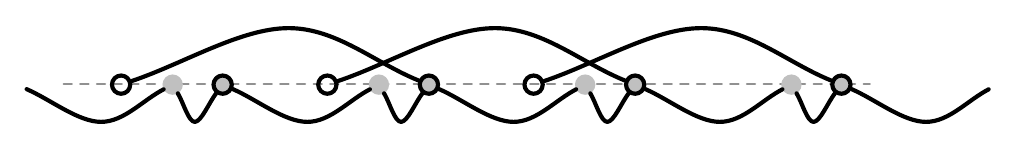}}
\caption{A minimally rigid 1d-periodic framework: (a) the underlying colored graph; (b) the 1d-periodic framework.
Bars are shown curved to avoid putting them on top of each other.}
\label{fig:1d-framework}
\end{figure}
We will give two arguments.  The first is geometric and does not generalize to the plane.  The
second uses (as we need to in the plane) \emph{infinitesimal rigidity} and relies on the natural
representations of $(1,1,k)$-graphs from \secref{11k-rep}.  As is the case for finite frameworks,
periodic direction networks are trivial objects in dimension one, so we do not develop them.

Because this is a ``warmup'' to indicate intuition, we will elide some details in the interest of
brevity.  Readers who are familiar with rigidity theory may wish to skip to \secref{directions}.

\subsection{Geometric proof}
Let $(\tilde{G},\varphi,\tilde{\bm{\ell}})$ be an abstract 1d-periodic framework.  In principle, to specify
a realization, we have to specify infinitely many points $x_i$: one for each vertex of $\tilde{G}$.  However,
the assumption that the $\Z$-action $\varphi$ has finite quotient means that there is really only
finite information present.  In particular, once we know:
\begin{itemize}
\item The location of one point in each $\Z$-orbit of vertices.
\item The real number $L$ representing $\Z$ by translations.
\end{itemize}
we can reconstruct the entire realization.

It is not hard to see that the continuity of the distance function implies that any connected
component of $\tilde{G}$ is rigid (\cite[Section 2.5]{G01} contains the details).  Thus
any connected component containing two vertices in the same $\Z$-orbit fixes $L$.  Such a
connected component is necessarily infinite, since $\tilde{G}$ is periodic.   \lemref{development-connectivity} then
implies that
the colored quotient must contain a cycle with $\Z$-rank one.  (\figref{1d-flexible-example} shows what
happens when this fails.)
\begin{figure}[htbp]
\centering
\subfigure[]{\includegraphics[width=.25\textwidth]{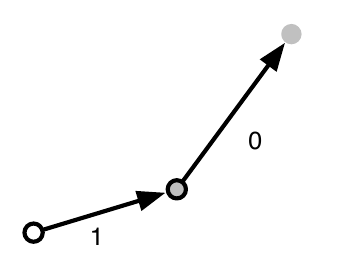}}
\subfigure[]{\includegraphics[width=.8\textwidth]{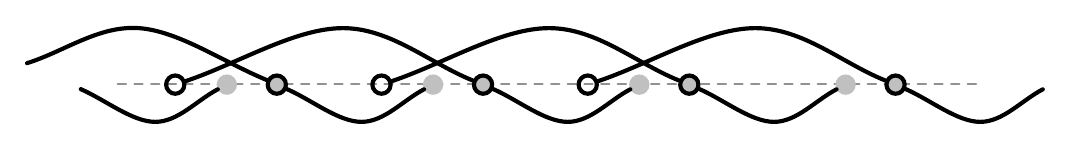}}
\subfigure[]{\includegraphics[width=.9\textwidth]{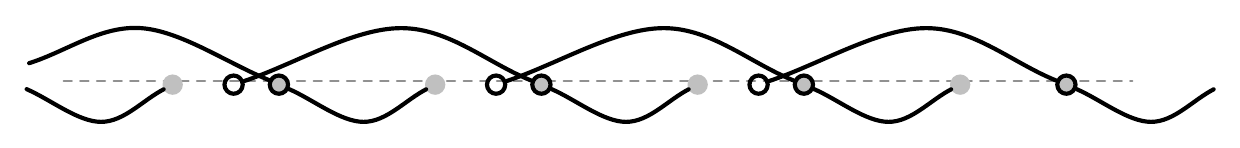}}
\caption{A flexible 1d-periodic framework: (a) the underlying colored graph; (b) the 1d-periodic frameowork;
(c) a non-trivial motion arising from changing the lattice representation $L$.}
\label{fig:1d-flexible-example}
\end{figure}
We also observe that if no infinite connected component of $\tilde{G}$ hits a vertex in each $\Z$-orbit, then
there are two orbits that can move independently of each other, leading to a flexible framework.  Thus any
rigid periodic framework on the line must contain an infinite connected component that hits every $\Z$-orbit of
vertices.  Sufficiency of the same condition is clear, so we have shown that a periodic framework on the
line is minimally rigid if and only if it's colored quotient is a $(1,1,1)$-graph.
\hfill $\qed$

\subsection{Proof via natural representations}\seclab{1d-full}
We now give a second proof of \theoref{1d-periodiclaman} that follows the general
approach we use to prove \theoref{periodiclaman}.

\vspace{0.75 ex}
\noindent \textbf{The continuous theory:} Rigidity and flexibility are determined by
the solution space to the infinite set of length equations:
\begin{itemize}
\item $|x_j - x_i| = \tilde{\ell_{ij}}$, for all edges $ij\in \tilde{G}$
\item $x_{\gamma\cdot i} = x_i + \gamma\cdot L$, for all $i\in V(\tilde{G})$ and $\gamma\in \Z$
\end{itemize}
where the unknowns are the points $x_i$ and the lattice representation $L$.

However, as noted above, since there is only finite information, we can identify this space with
the more tractable:
\begin{eqnarray*}
(x_j + \gamma_{ij}\cdot L - x_i)^2 = \ell^2_{ij} &
\text{for all colored edges $ij$ of the quotient graph $(G,\bgamma)$}
\end{eqnarray*}
We define the set of solutions to these equations to be the \emph{realization space} $\mathcal{R}(G,\bgamma)$
of the \emph{colored framework} $(G,\bgamma,\ell)$.  We note that since $\tilde{\bm{\ell}}$ had to assign the
same length to each $\Z$-orbit of edges, the colored framework is well-defined.  The \emph{configuration space}
$\mathcal{C}(G,\bgamma)$ of the colored framework is then defined to be the quotient
$\mathcal{R}(G,\bgamma)/\operatorname{Euc}(1)$ of the realization space by isometries of the line.

This formalism allows us to define rigidity: a realization of a 1d-periodic framework is \emph{rigid} when
it is isolated in the configuration space.

\vspace{0.75 ex}
\noindent \textbf{The infinitesimal theory:} The rigidity question, then, turns out to be one about the dimension
of the configuration space near a realization.  In the most general setting, this is a difficult question, but
at a smooth point, an adaptation of the arguments of Asimow and Roth \cite{AR78} show that a realization is
rigid if and only if the tangent space of the realization space is one-dimensional.

Taking the formal differential of the equations defining the realization space and dividing by two,
we obtain the system
\[
\bordermatrix{         &              &     i    &             &     j  &            &           L      \cr
& \dots        & \dots    &       \dots & \dots  & \dots      & \dots            \cr
ij &  0\dots      & -\eta_{ij}  & 0\dots0 & \eta_{ij} & 0\dots 0& \gamma_{ij} \eta_{ij} \cr
& \dots & \dots                & \dots & \dots               & \dots & \dots       }.
\]
where $\eta_{ij} = x_j +\gamma_{ij} L - x_i$, which we define to be the \emph{1d-rigidity matrix}.
The kernel of this 1d-rigidity matrix is identified with the tangent space
$T_{\vec p, L}(\mathcal{R}(G,\bgamma))$ of the realization space at the point $(\vec p, L)$.

\vspace{0.75 ex}
\noindent \textbf{Genericity and the combinatorial theory:}
Provided that the $\eta_{ij}$ are all non-zero, this matrix is just $\vec M_{1,1,1}(G,\bgamma)$ with the last column
discarded.  The condition for any of the $\eta_{ij}$ being zero is a measure-zero algebraic subset of $\mathbb{R}^{n+1}$,
which we define to be the \emph{non-generic} set of realizations.  If the $x_i$ and $L$ avoid the non-generic set,
then \lemref{111rep} implies that the 1d-rigidity matrix has corank one
if and only if $(G,\bgamma)$ contains a spanning $(1,1,1)$-graph, completing the proof.
\hfill $\qed$

\section{Periodic and colored direction networks}\seclab{directions}
We recall the following definitions from the introduction.  A \emph{periodic direction network}
$(\tilde G, \varphi,\tilde{\vec d})$ is an infinite multigraph $\tilde G$ with a
free $\mathbb{Z}^2$-action $\varphi$ by authormorphisms and an assignment of directions
$\tilde{\vec d} = (\tilde{\vec d}_{ij})_{ij\in E(\tilde G)}$ to the edges of $\tilde G$.

A \emph{realization} $\tilde G(\vec p,\vec L)$
of a periodic direction network is a mapping $\vec p$ of the vertex set $V(\tilde G)$ into
$\mathbb{R}^2$ and a matrix $\vec L\in \R^{2\times 2}$ representing
$\mathbb{Z}^2$ by translations of $\mathbb{R}^2$ such that:
\begin{itemize}
\item The representation $\Z^2 \to \R^2$ from $\vec L$ is equivariant with respect to the actions on $\tilde{G}$ and the plane;
i.e., $\vec p_{\gamma\cdot i} = \vec p_i + \vec L\cdot \gamma$ for all $i\in V(\tilde{G})$ and
$\gamma\in \Z^2$.
\item The specified edge directions are preserved by $\vec p$; i.e.,
$\vec p_j-\vec p_i=\alpha_{ij}\tilde{\vec d}_{ij}$ for all edges $ij\in E(\tilde G)$ and some
$\alpha_{ij}\in \mathbb{R}$
\end{itemize} An edge $ij$ is
\emph{collapsed} in a realization $\tilde G(\vec p)$ if $\vec p_i = \vec p_j$; a realization in
which all edges are collapsed is defined to be
a \emph{collapsed realization}, and a realization in which no edges are collapsed is \emph{faithful}.
Our main result on periodic direction networks is
\periodicparallel
In the next two sections we develop the tools we need, and then give the proof in
\secref{parallel}.

\subsection{Colored direction networks}
To study realizations of periodic direction networks we first reduce the problem to a
finite linear system.  A \emph{colored direction network} $(G,{\bm{\gamma}},\vec d)$
is defined to be a colored graph along with an assignment of directions to the edges.
A \emph{realization} $G(\vec p,\vec L)$ of the colored direction network
$(G,{\bm{\gamma}},\vec d)$ is a mapping $\vec p=(\vec p_i)_{i\in V(G)}$ into $\mathbb{R}^2$
such that:
\[
\vec p_j+\vec L\gamma_{ij}-\vec p_i=\alpha_{ij}\vec d_{ij}
\]
for some real number $\alpha_{ij}$.

An edge $ij\in E(G)$ is \emph{collapsed} in a realization
$G(\vec p,\vec L)$ if $\vec p_i=\vec p_j+\vec L\gamma_{ij}$; a realization with no
collapsed edges is defined to be \emph{faithful}.  A realization is {\em collapsed}
if all edges are collapsed.

The problems of periodic direction network realization and colored direction network
realization are equivalent.
\begin{lemma}\lemlab{dnrealizations}
Let $(\tilde G,\varphi,\vec d)$ be a periodic direction network.
Then the realizations of  $(\tilde G,\varphi,\vec d)$
are in bijective correspondence with the realizations of the colored direction
network $(G,{\bm{\gamma}},\vec d)$ on the quotient graph $(G,{\bm{\gamma}})$.  Furthermore, a
realization of $(\tilde G,\varphi,\vec d)$  is collapsed
if and only if the corresponding realization of $(G,{\bm{\gamma}},\vec d)$ is.
\end{lemma}
\begin{proof}
The proof is very similar to that of \lemref{dictionary}.
Any realization $G(\vec p,\vec L)$ of $(G,{\bm{\gamma}},\vec d)$ can be extended to a
$\tilde G(\vec p,\vec L)$ realization of $(\tilde G,\varphi,\vec d)$
via the $\mathbb{Z}^2$-action induced by ${\bm{\gamma}}$; in the other direction, a periodic
realization $\tilde G(\vec p,\vec L)$ induces a colored realization of
$(G,{\bm{\gamma}},\vec d)$ via the vertex representatives in $\tilde{G}$ of the vertices
of the quotient graph $G$.
\end{proof}

We define the \emph{colored direction network realization system} $\vec P(G,{\bm{\gamma}},\vec d)$ to be given by:
\[
\iprod{\vec p_j+\vec L\gamma_{ij}-\vec p_i}{\vec d^\perp_{ij}}=0\qquad \text{for all edges $ij\in E(G)$}
\]
The unknowns are the points $\vec p_i$ and the matrix $\vec L$; the given data are the edge
directions $\vec d_{ij}$.

\section{Properties of colored direction networks}\seclab{colored-realization}
We develop the properties of the system $\vec P(G,{\bm{\gamma}},\vec d)$ that we will need.

\subsection{Collapsed realizations of colored-Laman graphs}
Collapsed realizations of colored direction networks on colored-Laman graphs have a simple form: they
force the lattice representation to be trivial and put all the points on top of each other.
\begin{lemma} \lemlab{collapsedequivalence}
Let $(G, \bm{\gamma})$ be colored-Laman.  Then, a realization $G(\vec p, \vec L)$ of $(G, \bm{\gamma}, \vec d)$
is collapsed if and only if $\vec L=\begin{pmatrix} 0 & 0 \\ 0 & 0\end{pmatrix}$ and
$\vec p_i = \vec p_j$ for all $i, j \in V(G)$.
\end{lemma}
\begin{proof}
Summing the relations $\vec p_j + \vec L \gamma_{ij} - \vec p_i = 0$
over a cycle $C$ yields the equation $\vec L \rho(C) = 0$.  Since there are two
cycles with linearly independent $\rho(C)$, this implies that
$\vec L = \begin{pmatrix} 0 & 0 \\ 0 & 0\end{pmatrix}$.  The fact that
$\vec p_i = \vec p_j$ for all $i, j \in V(G)$ then follows from the connectedness of
colored-Laman graphs (which follows, for instance, by
\lemref{222decomp} and \lemref{periodiclamancontract}).
The converse is clear.
\end{proof}

\subsection{Translation invariance}
This lemma formalizes the geometric observation that translating any realization of a
colored direction network results in another realization.
\begin{lemma}\lemlab{invariance}
The set of solutions $(\vec p,\vec L)$ to
$\vec P(G,{\bm{\gamma}},\vec d)$ is invariant under translation of the points $\vec p_i$ and scaling of $(\vec p, \vec L)$.
\end{lemma}
\begin{proof}
Let $\vec t$ be a vector in $\mathbb{R}^2$ and $\lambda$ a scalar in $\R$.  Then
\[
\iprod{(\lambda \vec p_j+\vec t)+ \lambda \vec L\gamma_{ij} -
(\lambda \vec p_i-\vec t)}{\vec d^\perp_{ij}} = \lambda \iprod{\vec p_j+\vec L\gamma_{ij}-\vec p_i}{\vec d^\perp_{ij}}
\]
\end{proof}

\subsection{Relationship to the $(2,2,2)$-matroid}
Our main tool for moving back and forth between (geometric) colored direction networks and (combinatorial)
colored graphs is that the system $\vec P(G,{\bm{\gamma}},\vec d)$ is closely related to the
generic representation of the $(2,2,2)$-matroid.
\begin{lemma}\lemlab{realizationmat}
The solutions $(\vec p,\vec L)$ to the system $\vec P(G,{\bm{\gamma}},\vec d)$ are the $(\vec p,\vec L)$ satisfying
\[
\vec M_{2,2,2}(G,{\bm{\gamma}})(\vec p,\vec L)^{\mathsf{T}}=0
\]
\end{lemma}
\begin{proof}
Using the bilinearity of the inner product we get
\[
\iprod{\vec p_j+\vec L\gamma_{ij}-\vec p_i}{\vec d^\perp_{ij}} = \iprod{\vec p_j-\vec p_i}{\vec d^\perp_{ij}}+
\iprod{\vec L_1}{\gamma_{ij}^1\vec d_{ij}^\perp} + \iprod{\vec L_2}{\gamma_{ij}^2\vec d_{ij}^\perp}
\]
where the $\vec L_i$ are the columns of the matrix $\vec L$. In matrix form this is $\vec M_{2,2,2}(G,{\bm{\gamma}})$.
\end{proof}

\lemref{realizationmat} implies that we can determine the dimension of generic
colored direction network realization spaces using our results on natural representations
of the $(2,2,2)$-matroid.
\begin{lemma}\lemlab{realizationrank}
Let $(G,{\bm{\gamma}})$ be a colored graph with $n$ vertices, and $m$ edges.
The generic rank of the system $\vec P(G,{\bm{\gamma}},\vec d)$
(with the coordinates of $\vec p_1, \ldots, \vec p_n$ and entries of $\vec L$ as the unknowns)
is $m$ if and only if $(G,{\bm{\gamma}})$ is $(2,2,2)$-sparse.  In particular,
it is $2n+2$ if and only if $(G,{\bm{\gamma}})$ is a $(2,2,2)$-graph.
\end{lemma}
\begin{proof}
Apply \lemref{realizationmat} and then \cororef{M2rep}.
\end{proof}

\subsection{Genericity for colored direction networks}
Combining Lemmas \ref{lemma:realizationmat} and \ref{lemma:realizationrank} we see that the set of
directions for which the rank of $\vec P(G,{\bm{\gamma}},\vec d)$ is not predicted combinatorially by $(2,2,2)$-sparsity
is a measure-zero algebraic subset of $\mathbb{R}^{2m}$.
\begin{lemma}\lemlab{222genericity}
Let $(G,{\bm{\gamma}})$ be a $\mathbb{Z}^2$-rank $k$ $(2,2,2)$-$\mathbb{Z}^2$-graded-sparse colored graph with $n$ vertices,
$c$ connected components, and $m\le 2n+2k-2c$ edges.  The set of edge directions $\vec d$ such that the rank of $\vec P(G,{\bm{\gamma}},\vec d)$ is $m$ is the (open, dense)
complement of an algebraic subset of $\mathbb{R}^{2m}$.
\end{lemma}
\begin{proof}
By Lemmas \ref{lemma:realizationmat} and \ref{lemma:realizationrank},
the rank is $m$ unless $\vec d$ is a common zero of all the $m\times m$ minors of the matrix
$\vec M_{2,2,2}(G,{\bm{\gamma}})$, which is a nowhere-dense closed algebraic subset
of $\mathbb{R}^{2m}$.
\end{proof}

\section{Collapse of colored-Laman circuits}\seclab{collapse}

In this short section we prove the main technical lemmas we need for \theoref{periodicparallel}.

\subsection{Generic direction networks on colored-Laman circuits}
Generic colored direction networks on colored-Laman circuits (defined in \secref{cloning}) have very simple
realization spaces: all realizations are collapsed.
\begin{lemma}\lemlab{collapse}
Let $(G,{\bm{\gamma}})$ be a colored-Laman circuit with $n$ vertices, $c$ connected components, $\Z^2$-rank $k$,
and $m= 2n+2k-2c$ edges.  Furthermore, assume that $\vec P(G,{\bm{\gamma}},\vec d)$ has rank $2n+2k - 2c$ (this is
possible by \lemref{222genericity}).  Then all solutions of $\vec P(G,{\bm{\gamma}},\vec d)$ are collapsed.
\end{lemma}
\lemref{collapse} will follow from the following general fact about the subspace of collapsed realizations of any
colored direction network.
\begin{lemma}\lemlab{collapsed-dim}
Let $(G,{\bm{\gamma}})$ be a colored graph with $n$ vertices, $c$ connected components, $\Z^2$-rank $k$.  Any
direction network on the graph $(G,\bgamma)$ has a $(4-2k+2c)$-dimensional space of collapsed realizations.
\end{lemma}
The intuition behind this lemma is that we can freely select the position of one vertex in each connected component
which then determines the location of the rest of the vertices in that component, accounting for the $2c$
term.  Additionally, when the $\Z^2$-rank is non-zero, the representation $\vec L$ of the lattice is restricted if
we want to get a collapsed realization, giving the $4-2k$.
\begin{figure}[htbp]
\centering
\subfigure[]{\includegraphics[width=.2\textwidth]{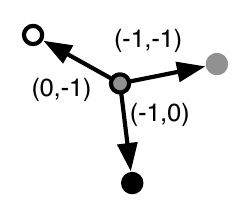}}
\subfigure[]{\includegraphics[width=.25\textwidth]{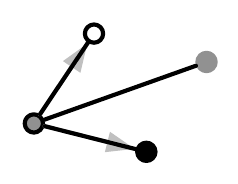}}
\subfigure[]{\includegraphics[width=.35\textwidth]{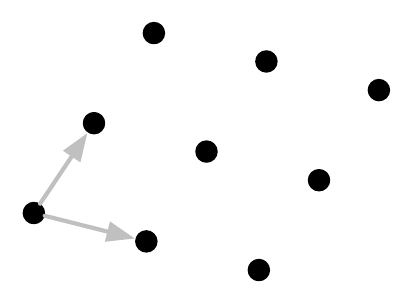}}
\caption{Constructing a collapsed realization of a tree: (a) the underlying colored graph; (b) the location of the
points in a colored realization; (c) in the development, we don't see any of the edges, since all the
vertex orbits are just translates of the same point set, reflecting the fact that the direction
condition is trivially met for a collapsed edge.}
\label{fig:collapse-tree-edges}
\end{figure}
\begin{proof}[Proof of \lemref{collapsed-dim}]
We first account for the $4$ parameters in the lattice representation $\vec L$.  Select the matrix $\vec L$
such that:
\[
\vec L\cdot\rho(C) = 0
\]
for every cycle $C$ in $G'$.  Under this condition, the number of free parameters in $\vec L$
is $4-2k$.

For now, assume that $G$ is connected, let $T$ be a spanning tree of of $G$.  For distinct vertices $v$
and $w$ in $G$, define $P_{vw}$ to be the path from $v$ to $w$ in $T$, and define $\sigma_{vw}\in \Z^2$
to be:
\[
\sigma_{vw} =  \left(\sum_{\substack{\text{$ij\in P_{vw}$ }\\ \text{traversed from $i$ to $j$}}}\gamma_{ij}\right) -
\left(\sum_{\substack{\text{$ij\in P_{vw}$} \\\text{traversed from $j$ to $i$}}}\gamma_{ij}\right)
\]

Select a root vertex $r$ and set $\vec p_r = (x_r,y_r)$ arbitrarily, and
then set $\vec p_i = \vec p_r - \vec L\cdot\sigma_{ri}$.  We check that all edges $ij$ in $G$ are collapsed.
If $ij$ is in the tree $T$, then $\gamma_{ij} = \sigma_{rj} - \sigma_{ri}$.  It then follows that
\[
\vec p_j - \vec p_i = \vec L\cdot\sigma_{ri} - \vec L\cdot\sigma_{rj} = -\vec L\cdot \gamma_{ij}
\]
so all the tree edges are collapsed.  (See \figref{collapse-tree-edges} for an example of the
construction.)
\begin{figure}[htbp]
\centering
\subfigure[]{\includegraphics[width=.3\textwidth]{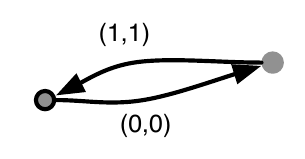}}
\subfigure[]{\includegraphics[width=.2\textwidth]{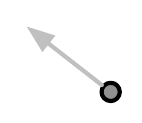}}
\subfigure[]{\includegraphics[width=.25\textwidth]{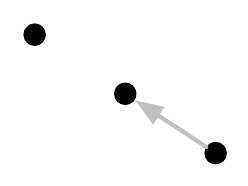}}
\caption{Constructing a collapsed realization of a $\Z^2$-rank $1$ cycle: (a) the underlying colored graph; (b) the location of the
points in a colored realization, with the two vertices on top of each other and the lattice representation degenerate;
(c) in the development, the lattice degenerates in the direction $(1,1)$
so that the cycle ``closes up'' and collapses.}
\label{fig:collapse-non-tree-edges}
\end{figure}

For non-tree edges $ij$, let $C_{ij}$ be the fundamental cycle of $ij$ with respect to $T$.
Using the identity $\rho(C_{ij}) - \gamma_{ij} = \sigma_{ri}  - \sigma_{rj}$ we compute
\[
\vec p_j - \vec p_i =  \vec L\cdot\sigma_{ri} - \vec L\cdot\sigma_{rj} = \vec L\rho(C_{ij}) - \vec L\cdot\gamma_{ij}
\]
and, since $\vec L\cdot\rho(C_{ij}) = 0$ (by construction), the edge $ij$ is collapsed as well.
(\figref{collapse-non-tree-edges} shows an example.)

The general case of the lemma follows from considering the connected components one by one.
\end{proof}

\lemref{collapse} follows nearly immediately from \lemref{collapsed-dim}.
\begin{proof}[Proof of \lemref{collapse}]
The hypothesis of the lemma is that the realization space is $(4-2k+2c)$-dimensional.  By
\lemref{collapsed-dim} the space of collapsed solutions has at least this dimension, so the
two coincide.
\end{proof}

\subsection{Collapsed edges and doubling an edge}
We now turn to the case in which the underlying colored graph of the colored direction network
is not $(2,2,2)$-colored.  In this case, collapsed edges can be given a combinatorial interpretation.

\begin{lemma}\lemlab{collapsecontract}
Let $(G,{\bm{\gamma}},\vec d)$ be a generic colored direction network, and let $ij$ be an edge of $E(G)$.
Suppose every solution $(\vec p,\vec L)$ of $\vec P(G,{\bm{\gamma}},\vec d)$
has $\vec p_i=\vec p_j+\vec L\gamma_{ij}$ (i.e., $ij$ is collapsed).  Then, $(\vec p,  \vec L)$
is a solution to $\vec P(G,{\bm{\gamma}},\vec d)$ if and only if it is a solution to
$\vec P(G+(ij)_c, {\bm{\gamma}},\vec d')$  for any extension $\vec d'$ of the
assignment $\vec d$ to $G+(ij)_c$.
\end{lemma}
\begin{proof}
Since every solution of $\vec P(G,{\bm{\gamma}},\vec d)$ has $\vec p_j +\vec L \gamma_{ij} - \vec p_i = 0$,
we can add a new constraint of the form $\langle \vec p_j +\vec L \gamma_{ij} - \vec p_i, (a, b) \rangle = 0$
without changing its solution set.  This is equivalent to a system of the
form $\vec P(G+(ij)_c, {\bm{\gamma}},\vec d')$ where $\vec d'$ is an extension of the
assignment $\vec d$ to $G+(ij)_c$.
\end{proof}

\section{Genericity and Proof of \theoref{periodicparallel}}\seclab{parallel}
We are nearly ready to prove \theoref{periodicparallel}.
\subsection{Genericity for colored-Laman direction networks}
The last technical tool we need is a description of the set of generic directions for direction networks
on colored-Laman graphs.
\begin{lemma}\lemlab{genericity}
Let $(G,{\bm{\gamma}})$ be a colored-Laman graph on $n$ vertices (and thus $m=2n+1$ edges).
The set of directions $\vec d\in \mathbb{R}^{4n+2}$
such that:
\begin{itemize}
\item $\vec P(G,{\bm{\gamma}},\vec d)$ has rank $2n+1$
\item For all edges $ij\in E(G)$, $\vec P(G+(ij)_c,{\bm{\gamma}},\vec d')$ has rank $2n+2$ for some $\vec d'$ extending $\vec d$
\end{itemize}
is the open, dense complement of an algebraic subset of $\mathbb{R}^{4n+2}$
\end{lemma}
\begin{proof}
Applying \lemref{222genericity}	 to $(G,{\bm{\gamma}})$ and each graph $(G+(ij)_c,{\bm{\gamma}})$
yields a finite set of nowhere dense algebraic subsets
of $\mathbb{R}^{4n+2}$ for which the statement of the lemma does not hold.
The union of these is algebraic and nowhere dense, as required.
\end{proof}

\subsection{Remark on genericity}
We remark that non-generic sets of directions come in two types:
\begin{itemize}
\item Those for which $\vec P(G,{\bm{\gamma}},\vec d)$ has rank less than $2n+1$
\item Those for which some $\vec P(G+(ij)_c,{\bm{\gamma}},\vec d')$ has rank less than $2n+2$
\end{itemize}
Both of these conditions are necessary for the proof of \theoref{periodicparallel}, and thus the genericity
assumption given here can't be weakened too much.  They also have slightly different geometric interpretations:
\begin{itemize}
\item If the rank of $\vec P(G,{\bm{\gamma}},\vec d)$ is not maximum, then there is a larger than expected space
of non-collapsed realizations preserving the given directions.  These additional degrees of freedom translate
to non-trivial infinitesimal motions of periodic frameworks via a standard trick from parallel redrawing.
\item The rank of $\vec P(G+(ij)_c,{\bm{\gamma}},\vec d')$ not increasing means that the given directions are not
realizable as part of the difference set of points in the plane, which implies collapsed edges even before
doubling.  Intuitively, the rank of the colored direction network system doesn't rise when doubling a collapsed
edge because there is no new constraint on its direction.
\end{itemize}

\subsection{Proof of \theoref{periodicparallel}}
Let $(G,{\bm{\gamma}})$ be a colored-Laman graph, and
select $\vec d$ as in \lemref{genericity}.
By \lemref{realizationrank}, $\vec P(G,{\bm{\gamma}},\vec d)$ has a $3$-dimensional solution space.
The set of collapsed solutions is two-dimensional by \lemref{collapsedequivalence}.  Hence, there is
a solution $\vec p=\hat{\vec p}$ and
$\vec L=\hat{\vec L}$ that is not collapsed, and by \lemref{invariance} we can assume that
$\hat{\vec p}_1 = (0, 0)$.  Any other solution with $\vec p_1 = (0, 0)$ is $(\hat{\vec p}, \hat{\vec L})$
up to scaling by some real number $\lambda$.

We suppose, for a contradiction, that some edge $ij$ is collapsed in $(\hat{\vec p}, \hat{\vec L})$.
Because all the realizations are scalings of $(\hat{\vec p}, \hat{\vec L})$,
$ij$ must be collapsed in all realizations. It follows from \lemref{collapsecontract} that
$\vec P(G,{\bm{\gamma}},\vec d)$ has the same solution space as
$\vec P(G+(ij)_c,{\bm{\gamma}},\vec d')$ where $\vec d'$ is chosen as in \lemref{genericity}.

The combinatorial \lemref{periodiclamancontract} implies that $(G+(ij)_c,{\bm{\gamma}})$ is $(2,2,2)$-colored.
By the hypothesis on $\vec d'$, from \lemref{genericity}, $\vec P(G+(ij)_c,{\bm{\gamma}},\vec d)$ has full
rank, and then \lemref{collapse} implies that all solutions of $\vec P(G+(ij)_c,{\bm{\gamma}},\vec d')$,
and thus $\vec P(G,{\bm{\gamma}},\vec d)$ are collapsed.  This contradicts
our assumption that $(\hat{\vec p}, \hat{\vec L})$ is not collapsed, proving that,
if $(G,{\bm{\gamma}})$ is colored-Laman and $\vec d$ is chosen generically as in \lemref{genericity}, then all
realizations with at least one non-collapsed edge are faithful.

As noted above, the realization space is three dimensional.  \lemref{invariance} shows that there is
a $2$-dimensional subspace of translations and, since the system is homogenous, scaling provides
an independent dimension of realizations.  This proves that the faithful realization is unique
up to translation and scale.

In the other direction, if $(G,{\bm{\gamma}})$ is not colored-Laman, then \cororef{M2rep} and \lemref{collapse},
applied to colored-Laman circuit supplied by \lemref{circuits} implies that some edge collapses.
\hfill $\qed$

\section{Periodic and colored rigidity}\seclab{continuous}
With \theoref{periodicparallel} proved, we return from the setting of direction networks to that of
bar-joint rigidity.  Sections \ref{section:continuous}--\ref{section:generic}
follows the same three-step outline used for the 1d-periodic case in \secref{1d-full}, going from
the continuous rigidity theory to the combinatorics of colored-Laman graphs and then (generically),
back again.  We start by recalling the definition of periodic frameworks from the introduction.

\subsection{Periodic frameworks}
A periodic framework is defined by a triple $(\tilde G, \varphi,\tilde{\bm{\ell}})$ where:
$\tilde{G}$ is a simple infinite graph; $\varphi$ is a free $\mathbb{Z}^2$-action on $\tilde{G}$ by automorphisms
such that the quotient is finite; and $\tilde{\bm{\ell}}=(\tilde{\ell_{ij}})$ assigns a
length to each edge of $\tilde G$.

A \emph{realization} $\tilde G(\vec p,\vec L)$ of a periodic framework $(\tilde G, \varphi,\tilde{\bm{\ell}})$
is defined to be a mapping $\vec p$ of the vertex set $V(\tilde G)$ into $\mathbb{R}^2$ and a representation
$\Z^2 \to \R^2$ encoded by a matrix $\vec L \in \R^{2 \times 2}$ (with $\R^2$ here viewed as translations) such that:
\begin{itemize}
\item the representation is equivariant with respect to the $\Z^2$-actions on $\tilde{G}$ and the plane; i.e.,
$\vec p_{\gamma\cdot i} = \vec p_i + \vec L\cdot \gamma$ for all $i\in V(\tilde{G})$ and
$\gamma\in \Z^2$.
\item The specified edge lengths are preserved by $\vec p$; i.e., $||\vec p_i - \vec p_j||=\tilde{\ell}_{ij}$ for all
edges $ij\in E(\tilde{G})$.
\end{itemize}
To be realizable, a periodic framework needs to assign the same length to edges in the same $\Z^2$-orbit, and
from now on we make this assumption, since we are interested in analyzing generic realizations.

\subsection{Periodic rigidity and flexibility}
The \emph{realization space} of a periodic framework is defined to be the algebraic set
$\Real(\tilde{G}, \varphi,\bm{\tilde{\ell}})$ of all realizations.  The group of $2$-dimensional
Euclidean isometries, $\operatorname{Euc(2)}$, acts naturally on
$\Real(\tilde{G}, \varphi,\bm{\tilde{\ell}})$; for $\phi \in \operatorname{Euc(2)}$ with
rotational part $\phi_0 \in \operatorname{Euc(2)}$, the action is
given by
\[\phi(\tilde{G}(\vec p, \vec L)) = \tilde{G}(\phi(\vec p), \phi_0 \circ \vec L)
\]
The \emph{configuration space}
$\Conf(\tilde{G}, \varphi,\bm{\tilde{\ell}})=\Real(\tilde{G}, \varphi,\bm{\tilde{\ell}})/\operatorname{Euc(2)}$
is defined to be the quotient of the realization space by Euclidean motions.
A realization $\tilde G(\vec p,\vec L)$ is \emph{rigid} if $\tilde{G}(\vec p,\vec L)$ is isolated in
the configuration space and \emph{minimally rigid} if it is rigid but ceases to be so when the
$\mathbb{Z}^2$-orbit of any edge $ij\in E(\tilde G)$
is removed.  Since $\Real(\tilde{G}, \varphi,\bm{\tilde{\ell}})$ is a subset of an infinite-dimensional space,
its topology merits some discussion.  The interested reader can refer to \cite[Appendix A]{MT10}\footnote{The reference
\cite{MT10} is a previous version of the present paper.}.

\subsection{Main theorem}
We can now state our main theorem:
\periodiclaman
The proof will make use of (technically simpler) colored frameworks, which we now define.
\subsection{Colored frameworks}
A priori, the realization space $\Real(\tilde{G}, \varphi, \bm{\tilde{\ell}})$ could be an unwieldy infinite dimensional object.
However, since $\tilde{G}/\Z^2$ is finite, the realization space is really finite dimensional.  We now make this precise
via the following definition.

A {\em $\Z^2$-colored framework} is defined as a triple $(G, {{\bm{\gamma}}}, \bm{\ell})$
where $(G, \bm{\gamma})$ is a $\Z^2$-colored graph and
$\bm{\ell}=(\ell_{ij})_{ij\in E(G)}$ is an assignment of lengths to the
edges of $G$.

A {\em realization} $G(\vec p,\vec L)$ of a $\Z^2$-colored
framework $(G, {{\bm{\gamma}}}, \bm{\ell})$ is an assignment $\vec p=(\vec p_i)_{i\in V(G)}$ of
points to the vertices of $G$ and a choice of matrix $\vec L \in \R^{2 \times 2}$ such that
for all $ij \in E(G)$ we have
\begin{equation}
\label{eq:lengths}\|\vec p_j+ \vec L\cdot {\gamma}_{ij} - \vec p_i \|^2=\ell_{ij}^2
\end{equation}
\begin{figure}[htbp]
\centering
\includegraphics[width=0.9\textwidth]{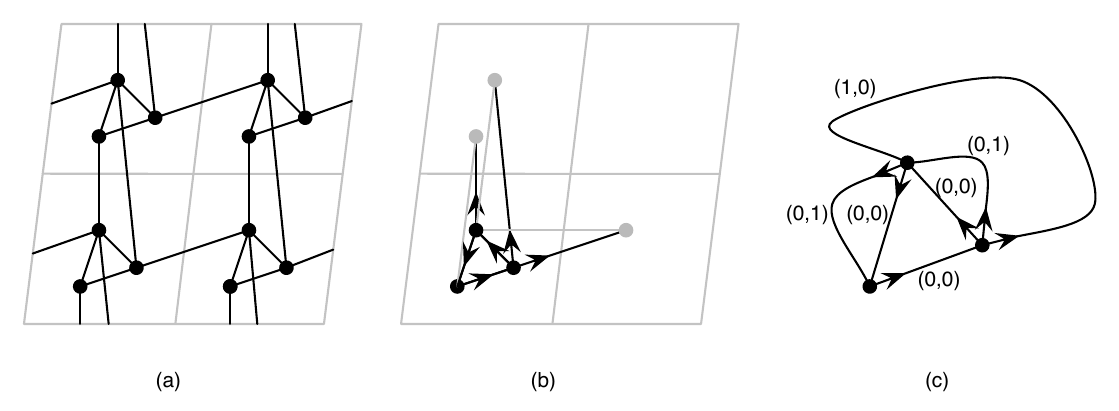}
\caption{The correspondence between periodic and colored frameworks: (a) a periodic framework;
(b) the associated colored framework; (c) the
underlying colored graph.}
\label{fig:framework-exampleX}
\end{figure}

It is clear from the definition that the realization space $\Real(G, {{\bm{\gamma}}}, \bm{\ell})$ is
naturally identified with a subvariety of $\R^{2n + 4} = (\R^2)^n \times \R^{2\times 2}$ where $n = |V(G)|$.
As with Lemma \ref{lemma:dictionary}, there is a dictionary between triples $(G, {{\bm{\gamma}}}, \bm{\ell})$
and triples $(\tilde{G}, \varphi, \bm{\tilde{\ell}})$ where $\tilde{G}$ is the development of $(G,\bgamma)$ and
$\tilde{\bm{\ell}}$ is obtained by assigning $\ell_{ij}$ to every edge in the fiber over $ij\in E(G)$.

\subsection{Continuous rigidity of colored frameworks}
As in the theory of finite (uncolored) frameworks in Euclidean space, if no vertex is
``pinned down,'' then there are always trivial motions of a
realization that arise from Euclidean isometries.  A realization is then rigid if these are
the only motions.  We now make the various notions of rigidity precise in the colored framework setting.

The isometry group $\operatorname{Euc(2)}$ of the Euclidean plane  acts naturally on $\Real(G, {{\bm{\gamma}}}, \bm{\ell})$.
For any $\phi \in \operatorname{Euc(2)}$, let $\phi_0 \in  \R^{2\times 2} $ be the rotational part.  Then the action
$$ \phi\cdot( \vec p_1, \dots, \vec p_n, \vec L) = (\phi(\vec p_1), \dots, \phi(\vec p_n), \phi_0 \cdot \vec L)$$ preserves
$\Real(G, {{\bm{\gamma}}}, \bm{\ell})$.  We define motions given by $\operatorname{Euc(2)}$ to be trivial, and we
define the {\em configuration space} $\Conf(G, {{\bm{\gamma}}}, \bm{\ell})$ to be
$\Real(G, {{\bm{\gamma}}}, \bm{\ell})/ \operatorname{Euc(2)}$.

Let $(G, {{\bm{\gamma}}}, \bm{\ell})$ be a $\Z^2$-colored framework.
A realization $G(\vec p, \vec L)$ of $(G, {{\bm{\gamma}}}, \bm{\ell})$ is {\em rigid} if
the corresponding point in $\Conf(G, {{\bm{\gamma}}}, \bm{\ell})$ is isolated.
Otherwise, it is {\em flexible}.  If $G(\vec p, \vec L)$ is rigid and is flexible
after the removal of any edge, we say $G(\vec p, \vec L)$ is {\em minimally rigid}.

\subsection{Equivalence of periodic and colored frameworks}
The following proposition can be obtained from \cite{BS10} by translating the arguments
into the setting of colored frameworks.
\begin{prop}[\frameworkdictionary][{\cite[Theorem 3.1]{BS10}}]\proplab{frameworkdictionary}
Let $(\tilde{G}, \varphi, \bm{\tilde{\ell}})$ be a periodic
framework and $(G, \bm{\gamma}, \bm{\ell})$ an associated $\Z^2$-colored graph.
There is a natural homeomorphism $\Psi: \Real(\tilde{G}, \varphi, \bm{\tilde{\ell}}) \to \Real(G, \bm{\gamma}, \bm{\ell})$
respecting the action of $\operatorname{Euc(2)}$.  In particular,
$\tilde{G}(\vec{\tilde{p}}, \vec L)$ is rigid if and only if
$\Psi(\tilde{G}(\vec{\tilde{p}}, \vec L))$ is rigid.
\end{prop}
\figref{framework-exampleX} shows the correspondence between periodic and colored
frameworks associated with the same colored graph.

\section{Infinitesimal colored rigidity}\seclab{infinitesimal}
We now introduce \emph{infinitesimal rigidity}, a linearization of
of the rigidity problem that is more tractable than the quadratic system of length equations.
The \emph{rigidity matrix} $\vec M_{2,3,2}(G,{\bm{\gamma}}, \vec p,\vec L)$ of a
colored framework is defined by the differential of the system \eqref{lengths}:
\[
\vec M_{2,3,2}(G, {{\bm{\gamma}}}, \vec p, \vec L)  = \bordermatrix{                   &              &     i       &             &     j      &            &    L_1   & L_2  \cr
& \dots & \dots & \dots & \dots & \dots & \dots  & \dots \cr
ij   &  \dots & -\eta_{ij} & \dots & \eta_{ij} & \dots & \gamma_{ij}^1 \eta_{ij} & \gamma_{ij}^2 \eta_{ij} \cr
& \dots & \dots & \dots & \dots & \dots & \dots & \dots  }
\]
where $\eta_{ij} = \vec p_j + \vec L\cdot\gamma_{ij} - \vec p_i$.  This matrix was first computed in \cite{BS10}.

The kernel of the rigidity matrix is defined to be the space of \emph{infinitesimal motions},
which spans the tangent space $T_{(\vec p,\vec L)}\Real(G,{\bm{\gamma}},\bm{\ell})$ of the realization
space at the point $(\vec p,\vec L)$.

It is shown in \cite{BS10} (and easy to check via direct computation) that the Lie algebra of $\operatorname{Euc(2)}$ always
induces a $3$-dimensional subspace of infinitesimal motions.  A realization $G(\vec p,\vec L)$ is defined to be
\emph{infinitesimally rigid} if the space of infinitesimal motions is $3$-dimensional and \emph{infinitesimally flexible} otherwise.
Infinitesimal rigidity is equivalent to the rigidity matrix having corank 3. Infinitesimal rigidity always implies rigidity,
but the converse holds only up to a nowhere dense set of \emph{non-generic} realizations, which we define below.

\subsection{Genericity for colored frameworks}
A realization $G(\vec p,\vec L)$ is defined to be \emph{generic} if the rank of the rigidity matrix is
maximized over all choices of $\vec p$ and $\vec L$; i.e.,
$\vec M_{2,3,2}(G,{\bm{\gamma}},\vec p,\vec L)$ achieves its generic rank at $G(\vec p,\vec L)$.
The important thing, for our purposes, is that the generic rank of the rigidity matrix
depends only on the underlying colored graph $(G,{\bm{\gamma}})$.

Thus, we define $(G, {{\bm{\gamma}}})$ to be
{\em generically rigid} (resp. {\em flexible}) if generic $G(\vec p, \vec L) \in \R^{2n+4}$ are rigid (resp. flexible).
Similarly define {\em generic infinitesimal rigidity} (resp. {\em flexibility}) of $(G, {{\bm{\gamma}}})$.  We define
$(G, {{\bm{\gamma}}})$ to be {\em generically minimally rigid} if generic $G(\vec p, \vec L)$ are minimally rigid.

The analogue of the following lemma for the non-periodic setting follows from the main theorem of
\cite{AR78}.  This result says intuitively that for \emph{generic} realizations,
continuous and infinitesimal rigidity have the same behavior.  With minor modifications,
the proofs of \cite{AR78} carry over to our setting \cite[Appendix A]{MT10}.
\begin{lemma}[\equivalentlem] \label{lemma:equivalent} A colored graph $(G, {{\bm{\gamma}}})$ is:
\begin{itemize}
\item Generically rigid if and only if it is generically infinitesimally rigid.
\item Generically flexible if and only if it is generically infinitesimally flexible.
\end{itemize}
\end{lemma}

\section{Generic periodic rigidity: Proof of the Main \theoref{periodiclaman}}\seclab{generic}
This completes the required background, and we are ready to prove our main result.

\subsection{Proof of \theoref{periodiclaman}}
Let $(G,{{\bm{\gamma}}})$ be a colored graph with $n$ vertices and $m=2n+1$ edges.
We may reduce to the case $m = 2n+1$ since if $m \neq 2n+1$, the colored graph $(G, \bm{\gamma})$ is neither colored-Laman nor
generically minimally rigid.
By \lemref{equivalent}, it suffices to verify that the generic rank of $\vec M_{2,3,2}(G,{\bm{\gamma}},\vec p,\vec L)$
is $2n+1$ if and only if $(G,{\bm{\gamma}})$ is colored-Laman, since removing any edge will
lead to a rigidity matrix with corank at least $4$.

First, suppose that $(G,{{\bm{\gamma}}})$ is not colored-Laman.  Then by \lemref{circuits}, it contains a
colored-Laman circuit $(G',\bgamma)$ on $n'> 2$ vertices, $c'$ components, rank $k'$ and $m' = 2n' + 2k' - 2c'$ edges.
This subgraph induces a submatrix $\vec M'$ of the same form as the rigidity matrix with $2n'+2k' -2c'$ rows and $2n'+4$
columns with non-zero entries.

We will show by contradiction that $\vec M'$ has rank less than $2n' + 2k' - 2c'$.  Suppose
$\vec M'$ has full rank.
Consider the direction network on $(G', \bm{\gamma})$ with directions $\vec d$ given by the edge directions
$\eta_{ij}$ of $G(\vec p, \vec L)$.  Since $G(\vec p, \vec L)$ is itself a realization of the direction
network, not all realizations are collapsed.  However, the matrix for the system
$\vec P(G,{\bm{\gamma}},\vec d)$ can be obtained from $M'$ by swapping and negating some columns.
Hence, the system $\vec P(G,{\bm{\gamma}},\vec d)$ has full rank, and by \lemref{collapse},
all solutions are collapsed, a contradiction.  Since $\vec M'$ has the same rank as the corresponding
$2n' + 2k' - 2c'$ rows in the rigidity matrix, $\vec M_{2,3,2}(G, {{\bm{\gamma}}}, \vec p, \vec L)$
must have a row dependency, and thus rank strictly less than $2n+1$.

Now we suppose that $(G, {{\bm{\gamma}}})$ is colored-Laman.
We will show it has full rank by an example.  Construct a generic (in the sense of \lemref{genericity})
direction network $(G,{\bm{\gamma}},\vec d)$ on $(G,{{\bm{\gamma}}})$.
By \theoref{periodicparallel}, for generic $\vec d$, this direction network has a unique,
up to translation and scaling, faithful realization $G(\vec p,\vec L)$.  Thus, for all $ij \in E(G)$,
there is $\alpha_{ij} \neq 0$ such that $\vec p_j + \vec L {{\bm{\gamma}}}_{ij}- \vec p_i = \alpha_{ij} \vec d_{ij}$.
By replacing $\vec d_{ij}^\perp$ with $\vec d_{ij}^\perp/\alpha_{ij}$ and swapping and negating
some columns in $\vec M_{2,2,2}(G, \bm{\gamma})$, we
obtain the rigidity matrix $\vec M_{2,3,2}(G, {{\bm{\gamma}}}, \vec p, \vec L)$.
Since all such operations do not affect the rank, $G(\vec p, \vec L)$ is infinitesimally rigid.
\hfill $\qed$

\subsection{Remarks}
Although we proved the rigidity \theoref{periodiclaman} from the direction network
\theoref{periodicparallel} algebraically, using matrix manipultions, there is a
more geometric way to view the argument.

Let $G(\vec p,\vec L)$ be a realization of a colored framework with underlying colored graph
$(G,\bgamma)$.  This realization induces a colored direction network $(G,\bgamma,\vec d)$, where the
direction $\vec d_{ij} = \vec p_j + \vec L\cdot \gamma_{ij} - \vec p_i$.
Now let $G(\vec p',\vec L')$ be another realization of $(G,\bgamma,\vec d)$.  By construction, we know that,
for every edge $ij$ in the colored graph $(G,\bgamma)$,
\[
\iprod{\vec p_j + \vec L\cdot \gamma_{ij} - \vec p_i}
{(\vec p'_j - \vec p_j + (\vec L - \vec L')\cdot\gamma_{ij} - (\vec p'_i - \vec p_i))^\perp} = 0
\]
In other words, the difference between $(\vec p,\vec L)$ and another realization
of the colored direction network $(G,\bgamma,\vec d)$ turned by $90$ degrees gives an infinitesimal motion of
the colored framework $G(\vec p,\vec L)$. The same fact for planar finite frameworks is classical.

\section{Conclusions and further directions}\seclab{conclusions}
We considered the question of generic combinatorial periodic rigidity in the plane, and, with
\theoref{periodiclaman}, gave a  complete answer.  To conclude we indicate some
consequences and potential further directions.

\subsection{Fixed-lattice frameworks}
Elissa Ross considered a specialization of the planar periodic rigidity problem
in which the lattice representation $\vec L$ is fixed.  She proved, in our language:
\begin{prop}[\fixedtoruslaman][\cite{R09,R11}]
\proplab{fixedtoruslaman}
Let $(\tilde G, \varphi,\tilde{\bm{\ell}})$ be a generic periodic framework and further suppose that
the lattice representation $\vec L$ is fixed, with $\vec L$ non-singular.  Then a generic
realization of $(\tilde G, \varphi,\tilde{\bm{\ell}})$ is minimally rigid if and only if
its quotient graph $(G,\bgamma)$:
\begin{itemize}
\item Has $n$ vertices and $m=2n-2$ edges.
\item Every subgraph $G'$ on $n'$ vertices and $m'$ edges with $\Z^2$-rank zero
satisfies $m'\le 2n'-3$.
\item Every subgraph $G'$ on $n'$ vertices and $m'$ edges satisfies $m'\le 2n' - 2$.
\end{itemize}
\end{prop}
We define a colored graph satisfying the properties of \propref{fixedtoruslaman} to be
a \emph{Ross graph}.  Ross graphs are related to colored-Laman graphs via the following
combinatorial equivalence.  (The colored-Laman graph in \figref{colored-laman-not-laman-spanning}
arises from the construction in \lemref{rossgraphs}.)
\begin{lemma}\lemlab{rossgraphs}
Let $(G,\bgamma)$ be a colored graph with $n$ vertices and $m$ edges.  Then $(G,\bgamma)$ is a
Ross graph if and only if for \emph{any} vertex $i\in V(G)$ adding three self-loops at vertex
$i$ with colors $(1,0)$, $(0,1)$, and $(1,1)$ yields a colored-Laman graph $(G',\bgamma)$.
\end{lemma}
\propref{fixedtoruslaman} can then be obtained from \theoref{periodiclaman}
and the observation that the rigidity matrix of the augmented graph $(G',\bgamma)$ has the form
\[
\vec M_{2,3,2}(G',{\bm{\gamma}}', \vec p, \vec L) =
\begin{array}{c}         \;\;\;\;      i    \;\;\;\;             j             \;\;\;\;\;\;\;\;\;\;\;   L_1   \;\;\;\;\;\;\;\;\;\;\; L_2\;\;\; \\
\left(		\begin{array}{c}   \\          \vec M_{2,3,2}(G, {\bm{\gamma}}, \vec p, \vec L)  \\    \\ \hline
\begin{array}{ccc} 0\dots\dots0\;\; 	& \vec L_1                           & \vec L_1  \\
0\dots \dots0\;\; 		& \vec L_2                           & \vec L_2  \\
0\dots \dots0\;\; 		& \vec L_1 +\vec  L_2                & \vec L_1 +\vec  L_2
\end{array}
\end{array}
\right)
\end{array}
\]
which implies that any infinitesimal motion acts trivially on the lattice representation $\vec L$.

\subsection{Crystallographic rigidity: other symmetry groups}
Our Main \theoref{periodiclaman} does not close the field of Maxwell-Laman-type Theorems for
planar frameworks with forced symmetry.  Perhaps the most natural question raised by the
present work is whether similar results are possible when $\Z^2$ is replaced by another
crystallographic group.

\subsection{Periodic parallel redrawing and scene analysis}
We introduced periodic direction networks with the goal of proving a characterization of
generic infinitesimal periodic rigidity, and thus have focused narrowly
on the properties needed for that purpose.  However, as discussed in the introduction,
there is a more general theory of \emph{parallel redrawing} and \emph{scene analysis},
which relate finite direction networks and frameworks to projections of polyhedral scenes
\cite[Sections 4 and 8]{W96}.  Determining the extent to which these theories generalize
to the periodic case would be very interesting.

\subsection{Group-graded sparsity and algorithmic periodic rigidity}
We introduced and studied two families of colored graphs: colored-Laman graphs and
$(2,2,2)$-graphs.  These are matroidal and, via general augmenting path algorithms
for matroid union, recognizable in polynomial time.

Two combinatorial questions that arise are:
\begin{itemize}
\item Is there a more general theory of matroidal hereditary sparsity for $\Z^d$-colored graphs?
\item Are there cleaner, more efficient algorithms for recognizing colored-Laman and $(2,2,2)$-graphs?
\end{itemize}
For finite frameworks, the answers to both of these questions are affirmative \cite{LS08}.

\subsection{Passing to sub-lattices}
Elissa Ross mentions the following conjecture, which relates to the
example from \secref{finite-index}.
\begin{conj}[{\cite[Conjecture 8.2.8]{R11}}]\conjlab{bob}
Let $\tilde{G}(\vec p,\vec L)$ be an infinitesimally rigid periodic framework with periodic
graph $(\tilde{G},\varphi)$.  Let $\Lambda< \Z^2$ be any sub-lattice, and define $(\tilde{G},\varphi')$ to be the
periodic graph obtained by replacing the $\Z^2$-action $\varphi$ with the induced $\Lambda$-action $\varphi'$.
Then $\tilde{G}(\vec p,\vec L)$ is an infinitesimally
rigid realization of the induced abstract periodic framework on $(\tilde{G},\varphi')$.
\end{conj}
Informally, what this conjecture says is that a generic, rigid periodic framework remains so even if
we enlarge the class of allowed motions by relaxing the periodicity constraint to hold only on a sub-lattice.
Geometrically, this means just expanding the fundamental domain of the $\Z^2$-action on the plane induced by
$\vec L$.
\begin{figure}[htbp]
\centering
\subfigure[]{\includegraphics[width=.25\textwidth]{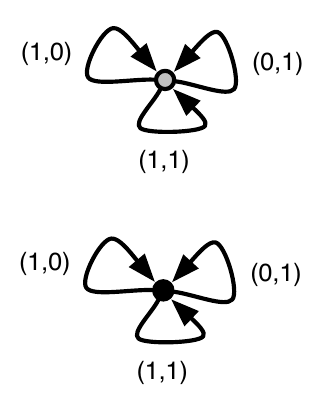}}
\subfigure[]{\includegraphics[width=.4\textwidth]{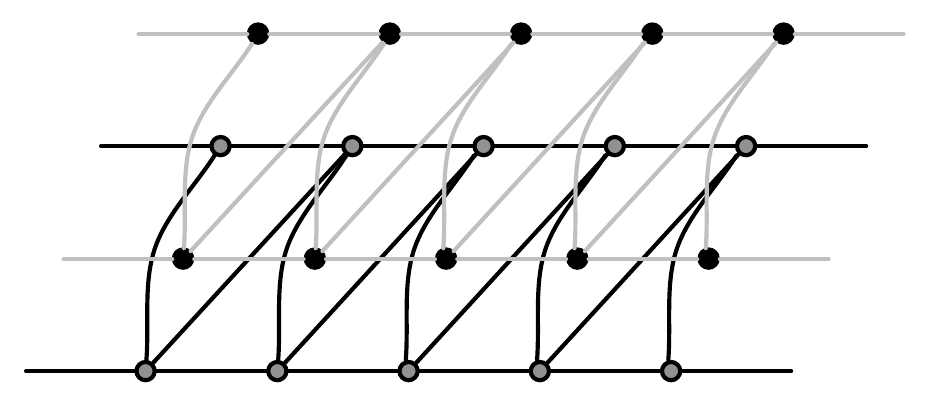}}
\caption{Restricting the $\Z^2$-action on the periodic graph from \figref{colored-laman-finite-index} (b)
to the sub-lattice generated by $(1,0)$ and $(0,2)$: (a) the resulting colored graph;
(b) the development, with connected components indicated by color.  Since the black points can
translate independently of the gray ones, any associated framework has at least these
non-trivial degrees of freedom.}
\label{fig:colored-laman-bad-sublattice}
\end{figure}
The example in \figref{colored-laman-finite-index} shows that \conjref{bob} is false, even
in a, much weaker, combinatorial version: if we take the sub-lattice to be the one generated by
$(1,0)$ and $(0,2)$, we get the periodic framework and associated colored graph in
\figref{colored-laman-bad-sublattice}. It is easy to see that the two connected components
can translate independently of each other, and that a maximal colored-Laman-sparse
subgraph is simply one of the connected components.

This counterexample generalizes.  Suppose that $(G,\bgamma)$ is a colored-Laman graph with
$n$ vertices.  The operation of passing to a sub-lattice $\Lambda$ corresponding to an index $\ell$
subgroup of $\Z^2$ in the associated periodic framework means, in combinatorial terms,
passing to an $\ell$-sheeted cover $(G^*,\bgamma^*)$ of the colored graph $(G,\bgamma)$.  Thus $G^*$
has $\ell n$ vertices and $2\ell n+\ell$ edges.

On the other hand, if $\rho(G,\bgamma)$ generates a finite index subgroup $\Gamma < \Z^2$, and we take
the corresponding sub-lattice $\Lambda$, then, by \lemref{development-connectivity}, $G^*$ has at least
two connected components, and thus any colored-Laman-sparse subgraph of $(G^*,\bgamma)$
can have at most $2\ell n + 1 - 2$ edges  which is too few to be a colored-Laman graph.  Repeating the same
construction, but with $(G,\bgamma)$ a subgraph of a colored-Laman graph $(H,\bgamma)$ with
$\rho(H,\bgamma) = \Z^2$, we see that the colored graph cover $(H^*,\bgamma)$ corresponding to the
bad sub-lattice $\Lambda$ need not be disconnected.

It would be interesting to resolve the following combinatorial question about colored graphs, which is a
kind of ``doubly generic'' version of \conjref{bob}.
\begin{question}
Let $(G,\bgamma)$ be a colored-Laman graph.  Does a generic finite-sheeted cover of $(G,\bgamma)$
that arises from passing to a sub-lattice in the development have a spanning subgraph that is colored-Laman?
\end{question}
We leave the meaning of generic intentionally vague, but it seems plausible that there are a finite
number of maximal ``bad'' sub-lattices to avoid.

\bibliographystyle{plainnat}

\end{document}